\documentclass[11pt, reqno]{amsart}

\usepackage{amssymb}
\usepackage{amsthm}
\usepackage{amsmath}
\usepackage{latexsym}
\usepackage{tikz}

\newtheorem{theorem}{Theorem}[section]
\newtheorem{proposition}[theorem]{Proposition}
\newtheorem{lemma}[theorem]{Lemma}

\newtheorem{remark}[theorem]{Remark}
\newtheorem{definition}[theorem]{Definition}

\makeatletter \@addtoreset{equation}{section} \makeatother
\renewcommand{\theequation}{\thesection.\arabic{equation}}

\newcommand{\R}{\mathbb{R}}

\newcommand{\be}{\begin{equation}}
\newcommand{\ee}{\end{equation}}
\newcommand{\bes}{\begin{eqnarray}}
\newcommand{\ees}{\end{eqnarray}}

\def\T{\mathcal{T}}

\def\bega{\begin{array}}
\def\enda{\end{array}}
\def\begi{\begin{itemize}}
\def\endi{\end{itemize}}
\usepackage[left=30 mm, right=30 mm,top=30 mm, bottom=30 mm]{geometry}

\begin{document}

\title[Solutions  including wave breaking for a class of equations]{\bf Existence and regularity for global solutions including breaking waves from Camassa-Holm and Novikov equations to  $\lambda$-family equations}

\author[G. Chen]{Geng Chen}
\address{Geng Chen \newline
Department of Mathematics, University of Kansas, Lawrence, KS 66045, USA.}
\email{gengchen@ku.edu}

\author[Y. Shen]{Yannan Shen}
\address{Yannan Shen \newline
Department of Mathematics, University of Kansas, Lawrence, KS 66045, USA.}
\email{yshen@ku.edu}

\author[S. Zhu]{Shihui Zhu}\footnote{Corresponding author: Shihui Zhu, E-mail: shihuizhumath@163.com}
\address{Shihui Zhu \newline
Department of Mathematics, Sichuan Normal University, Chengdu, Sichuan 610066, China.}
\email{shihuizhumath@163.com}


\maketitle

\begin{abstract}
In this paper,  we prove the global existence of H\"older continuous solutions for the Cauchy problem of a family of partial differential equations, named as $\lambda$-family equations, where $\lambda$ is the power of nonlinear wave speed.  The  $\lambda$-family equations include Camassa-Holm equation ($\lambda=1$) and Novikov equation  ($\lambda=2$) modelling water waves,  where solutions generically form finite time cusp singularities, or in another word,  show wave breaking phenomenon.  The global energy conservative solution we construct is H\"older  continuous with exponent $1- \frac{1}{2\lambda}$. The existence result also paves the way for the future study on uniqueness and Lipschitz continuous dependence.
\end{abstract}
\bigskip

\textit{Key Words: global existence, water wave equations, conservative solution, cusp singularity.}
%
\section{Introduction}

In this paper, we consider  a family of equations  given in the following form:
\be\label{E0} u_t-u_{txx}+(\lambda+2)u^{\lambda}u_x=(\lambda+1)u^{\lambda-1}u_xu_{xx}+u^{\lambda}u_{xxx}, \ee
where $u:=u(t,x): \mathbb{R}^+\times \mathbb{R}\rightarrow \mathbb{R}$ and $\lambda\geq 1$. For convenience, we name \eqref{E0} as the $\lambda$-family equations.

First,  let's introduce the physical backgrounds of \eqref{E0}. Among all equations, two special cases of (\ref{E0}) are of particular importance in studying water waves.
\begin{itemize}
\item[i.]  Camassa-Holm equation: (\ref{E0}) with $\lambda=1$
\be\label{CH}u_t-u_{txx}+3u u_x-2 u_xu_{xx}-uu_{xxx}=0.\ee
\item[ii.]
 Novikov equation:  (\ref{E0}) with $\lambda=2$
 \be\label{N}u_t-u_{txx}+4u^2 u_x-3 uu_xu_{xx}-u^2u_{xxx}=0.\ee
\end{itemize}
One can also find \eqref{E0} from a more general class of equations
\be\label{E0b} u_t-u_{txx}+(b+1)u^{\lambda}u_x-bu^{\lambda-1}u_xu_{xx}-u^{\lambda}u_{xxx}=0,\quad u_0(x)=u(0,x),\ee
when $b=\lambda+1$.  Especially,  when $\lambda=1$,  the equation \eqref{E0b} is sometimes called as the $b$-family equations or $b$-equations \cite{Anco,Panos},  including the Degasperis-Procesi equation when $b=3$.



As a representative equation in the form of \eqref{E0}, the Camassa-Holm equation (\ref{CH})
was originally derived as a model for surface waves, and has been studied extensively in the past three decades because solutions of this equation enjoy many interesting properties: infinite conservation laws and complete integrability \cite{CH1993, FF1981}, existence of peaked solitons and multi-peakons \cite{CH1993, cht}, break down of classical solutions \cite{M1998,CE1998,well}, and formation of finite time breaking waves. Here breaking waves mean that solutions remain bounded while their slope becomes unbounded in finite time. In particular, breaking waves are commonly observed in the ocean and are important for a variety of reasons. We refer the reader to two simple examples discussed in \cite{CCCS}, both of which include two interacting peakons and formation of finite time cusp singularity (breaking wave), for Camassa-Holm and Novikov equations.

The integrable Novikov equation (\ref{N}), first derived in \cite{N2009}, can be regarded as one of the modified Camassa-Holm equations with cubic nonlinearity. The general equation \eqref{E0} includes other equations with higher order nonlinearity.

Due to the formation of singularities of the strong solutions, it becomes imperative to consider weak solutions. Especially if one considers global solution,  in order to go beyond the breaking wave, it is natural to consider  H\"older continuous solution, for example, $H^1$ solution for Camassa-Holm equation. Let's focus on the Cauchy problem from now on. The global existence of ($H^1$) weak solutions for Camassa--Holm equation  has been established by two  frameworks: one applies the method of vanishing viscosity (see \cite{XZ2000}) and the other one is to introduce a new semilinear system on new characteristic coordinates (see \cite{BC2, BC3}).
Especially, the existences of energy conservative and dissipative solutions have  been established by Bressan-Constantin,  in \cite{BC2} and \cite{BC3}, respectively. The existence of $\alpha$-dissipative solution, which has partial energy dissipation, has later been established in \cite{GHR}.

The uniqueness and Lipschitz continuous dependence of global $H^1$ solution for Camassa-Holm equation have been established after the existence theory. To select a unique solution from conservative, dissipative and $\alpha$-dissipative solutions, it is necessary to apply an additional proper criteria for singling out an admissible weak solution, similar as adding an entropy condition for hyperbolic conservation laws. In \cite{BCZ2015}, the authors first proved the uniqueness of energy conservative solution which satisfies the energy conservation condition for weak solution. In \cite{CCM}, one can find an interesting uniqueness result for dissipative solution. The Lipschitz continuous dependence is a more involved problem, since the solution flow is not Lipschitz continuous under the $H^1$ distance. In \cite{BF,CCCS, GHR2011,GHR2013}, one can find Lipschitz continuous dependence results for the solution constructed in \cite{BC2} under a new optimal transport metric.

The existence, uniqueness, Lipschitz continuous dependence and generic regularity of global ($W^{1,4}$) weak solution for Novikov equation have been established in \cite{CCL}. However, the global well-posedness and regularity of  H\"older continuous solutions for general systems \eqref{E0} and \eqref{E0b} are still wide open. An apparent difficulty is that there are not many energy laws to use when we study \eqref{E0} and \eqref{E0b}. In fact, although Camassa-Holm and Novikov equations are both integrable systems, equations in the form of \eqref{E0} are not generally integrable. In this paper, we will prove the existence of global conservative solution, including cusp singularity, for \eqref{E0}.

Another major motivation of this paper is to study the regularity of global H\"older continuous solutions for equations \eqref{E0} with different $\lambda$.
Actually, a very interesting observation on existing results for Camassa-Holm and Novikov equations
shows that these two equations enjoy global well-posedness theories on {\em different} spaces. More precisely, the unique conservative H\"{o}lder continuous solution is: $C^{0,\frac12}$ for Camassa-Holm equation (\ref{CH}), \cite{BC2}; and $C^{0,\frac34}$ for Novikov equation (\ref{N}), \cite{CCL}.

A very natural and interesting problem is to study how the regularity of solution for (\ref{E0}) changes with respect to $\lambda$.  Briefly speaking, in this paper, we will construct energy conservative H\"{o}lder continuous solution with exponent $1- \frac{1}{2\lambda}$ for (\ref{E0}). This result precisely shows how the regularity of solution changes with respective to $\lambda$, so it unveils an intrinsic relation between Camassa-Holm equation (linear wave speed), Novikov equation (quadratic wave speed) and the general equation \eqref{E0} (wave speed $u^{\lambda}$). We also believe that the existence theory will pave the way for future studies on uniqueness, Lipschitz continuous dependence issues and other type of solutions such as the dissipative solution.



Now we give more details on our ideas and results.

\subsection{Notations, basic setups and energy}
We first introduce some notations and basic equations derived from the smooth solution. We define the function $p(x)=\frac 12 e^{-|x|}$, $x\in \mathbb{R}$,
where  $p-p_{xx}=\delta(x)$ with the Dirac function $\delta$, and,
for any smooth function $u(x)$,
$$(1-\partial_x^2)^{-1}u=p*u(x):=\int_{-\infty}^{+\infty}p(x-y)u(y)dy,\quad \hbox{and}\quad
p*(u-u_{xx})=u.$$
Using the properties of  $p(x)$, we can express  equation (\ref{E0}) in the following form
\be\label{E}
u_t+u^{\lambda}u_x+P_x+Q=0,
\ee
where the singular integral operators $P$ and $Q$ are defined by
\be\label{PQ}P:=p*[(\lambda-\frac{1}{2})u^{\lambda-1}u_x^2+ u^{\lambda+1}] \ \ {\rm and}\ \ Q:=\frac{\lambda-1 }{2}\,p*[ u^{\lambda-2}u_x^3].\ee
Differentiating  (\ref{E}) with respect to $x$, we get
\be\label{E-2}
\begin{array}{lll}u_{t x}+u^{\lambda}u_{xx}+\frac{1}{2}u^{\lambda-1} u_x^2- u^{\lambda+1}+P+Q_x=0.
 \end{array}
 \ee
Multiplying \eqref{E} by $u$ and \eqref{E-2} by $u_x$, one can easily prove the energy conservation law
\be\label{def_E0}
 \mathcal{E}(t):=\int_{\mathbb{R}} (u^2(t,x)+u_x^2(t,x))dx=\mathcal{E}(0),
\ee
for any smooth solutions.
\subsection{A general class of equations}
Before discussing the detail analysis, we would like to introduce a more general class of equations to show why the change of regularity with respect to $\lambda$ is an interesting topic to study.

The equation \eqref{E-2}-\eqref{PQ} is a special example of the following equation:
\be\label{spe}
	u_{tx}+f'(u)\, u_{xx}+\frac{1}{2\lambda} f''(u)\, (u_x)^2=g(u,u_x)\,. 
\ee
with constant parameter $\lambda\geq\frac{1}{2}$. We require $g(u,u_x)$ to be bounded by $\|u\|_{H^1}+\|u\|_{W^{1,2\lambda}}$.
More intuitively,  equation \eqref{spe} can be formally written as
\be\label{spe2}
u_{tx}+(f'(u))^{1-\frac{1}{2\lambda}}\, \bigl[(f'(u))^{\frac{1}{2\lambda}}\,u_{x}\bigr]_x=g(u, u_x)\,.
\ee
Equation \eqref{spe} includes several important and interesting models when $\lambda$ takes different values.  
\begin{itemize}
\item $\lambda=\frac{1}{2}, \ g(u)=0$: Scalar hyperbolic conservation law.
Shock waves (discontinuity) form even when initial data are smooth, and the BV solution exists \cite{D,Glimm}.\vspace{.2cm}
\item $\lambda=\frac{1}{2}, \ f(x)=\frac{1}{2}u^2,\ g(u)=u$: Ultra short pulse equation from nonlinear optics. The singularity is widely believed to be discontinuity, see singularity formation results in \cite{LPS}. BV existence has been established in \cite{CR}. \vspace{.2cm}
\item $\lambda=1,\ g(u)=0$: A general equation including Hunter-Saxton equation ($f(u)=\frac{1}{2}u^2$), which is a simplified model from nematic liquid crystals. The global $C^{0,\frac{1}{2}}$ solution including cusp singularity exists, \cite{BZZ,HS,HZ95a}.
\vspace{.2cm}
\item  $f(u)=\frac{u^{\lambda+1}}{\lambda+1}$ with $g(u)=u^{\lambda+1}-P-Q_x$: The $\lambda$-equation \eqref{E-2}-\eqref{PQ} including Camassa-Holm equation when $\lambda=1$ ($C^{0,\frac{1}{2}}$ solution \cite{BC}) and Novikov equation when $\lambda=2$ ($C^{0,\frac{3}{4}}$ solution \cite{CCL}).
 \end{itemize}
 
In \cite{CS2015}, the authors considered equation \eqref{spe} with $g=0$,
and showed the existence of global H\"older continuous weak solutions with exponent $1-\frac{1}{2\lambda}$. In this paper, we will prove the existence of global H\"older continuous weak solutions with exponent $1-\frac{1}{2\lambda}$ for  \eqref{E-2}-\eqref{PQ}, i.e. \eqref{spe} with $g(u)=u^{\lambda+1}-P-Q_x$, which includes some interesting water wave models.

By proving this result, we shows that the quasilinear wave type equation \eqref{spe2}, enjoys better regularity when ``more'' wave speed $c(u)$
 ``stays'' out of the parentheses. Similar trend also holds for quasilinear wave equations, including the p-system (BV solution) and variational wave equation ($C^{0,\frac{1}{2}}$ solution), see \cite{CS2015}.
%
%
%
%
%

%


\subsection{Key idea}
By the energy law \eqref{def_E0}, it seems natural to find a solution $u(\cdot,t)$ in $H^1$, similar as for Camassa-Holm equation ($\lambda = 1,\ Q = 0$). However, this is not true with the presence of $Q$ when $\lambda\neq 1$.
In fact, the energy $\mathcal E(t)$ is not enough to control the cubic nonlinearity $u_x^3$ in $Q$. For Novikov equation, with $\lambda=2$, another energy conservation law including $u_x^4$ is available. Using both lower order (on $u_x^2$) and higher order  (on $u_x^4$) energy conservation laws,  in \cite{CCL} Chen-Chen-Liu proved the global existence of solution $u(\cdot,t)\in W^{1,4}$  (or $C^{0,3/4}$ by Sobolev embedding), using the characteristic method framework first established in \cite{BC2}.

For \eqref{E0}, in order to find the appropriate H\"older space, it would be natural to first search for an energy conservation law whose energy density includes $(u_x)^i$ with $i\geq 3$. However, such an energy conservation law is not available.
Different from Camassa-Holm and Novikov equations, equation \eqref{E0} is in general not an integrable system , so there are not many energy laws to use.

However, a key observation of this paper is that   by (\ref{E-2}) we find an energy balance law
\be\label{E-3}\begin{array}{lll}&(u_x^{2\lambda})_t+(u^{\lambda}u_x^{2\lambda})_x=(u^{\lambda+1}
 - P-Q_x)2\lambda u_x^{2\lambda-1}.
 \end{array}\ee
On the right hand side of this equation, we need to control the cubic nonlinearity in $Q$. Notice that.  as long as $\lambda \geq 2$, we have
\be\label{i1}
(\int |u_x|^3\, dx)\, (\int |u_x|^{2\lambda-1}\, dx)\leq (\int |u_x|^2\, dx)\, (\int |u_x|^{2\lambda}\, dx)
\ee
then we may use \eqref{E-3} to control $u_x^{2\lambda}$ . In this paper, we always call \eqref{E-3} the higher order energy balance law comparing to the lower order energy conservation law \eqref{def_E0}. As we mentioned,  for Novikov equation, in \cite{CCL}, two energy conservation laws for $u_x^2$ and $u_x^4$ are both used to run an existence proof. Here instead we will  prove the existence of weak solution using a higher order energy balance law \eqref{E-3} combing with a lower order energy conservation law on $\mathcal E(t)$.


\subsection{Main result}

In this paper, we impose the initial data
\be\label{Ei}u(0,x)=u_0(x)\in  H^1{(\mathbb{R})}\cap W^{1,2\lambda}{(\mathbb{R})}.\ee
This space is consistent with two energy laws. Since singularity will generically formed in finite time \cite{BC2,CCCS},  for the general Cauchy problem,  solution is generically not smooth even when the initial data are smooth, although some special global classical solutions might be available.

We first define the weak solution for Cauchy problem (\ref{E}) (\ref{Ei}).

\begin{definition}\label{def1}
We say $u=u(t,x)$ is a weak solution of  the Cauchy problem (\ref{E}) (\ref{Ei}) (or (\ref{E0}) (\ref{Ei}) )  if it satisfies
\begin{itemize}
\item[(i)] For any fixed $t\geq 0$, $u(t,\cdot)\in H^1{(\mathbb{R})}\cap W^{1,2\lambda}{(\mathbb{R})}$.
The map $t\mapsto u(t,\cdot)$ is Lipschitz continuous under $L^{2\lambda}(\mathbb{R})$ distance.

\item[(ii)] Solution $u=u(t,x)$ satisfies initial condition  (\ref{Ei})  in $L^{2\lambda}(\mathbb{R}),$ and
\be\label{nv_weak}
\iint_{\Gamma}
	\left\{-u_x\,  \phi_t -u_x u^{\lambda} \phi_x  +[(\frac{1}{2}-\lambda)u^{\lambda-1}u_x^2-u^{\lambda+1} + P+Q_x]  \phi \right\}\, dx\, dt+\int_{\mathbb{R}}(u_{0})_{x}\phi(0,x)\,dx=0
\ee
for every test function $\phi\in C_c^1(\mathbb{R}^+\times \mathbb{R})$.

\end{itemize}
\end{definition}

The main existence result on energy conservative solution is stated in the following Theorem.
\begin{theorem}\label{main}
Let $\lambda=1$ or $\lambda\geq 2$. 
Suppose that  $u_0\in H^1(\mathbb{R})\cap W^{1,2\lambda}(\mathbb{R})$ is an absolute continuous function on $x$. Then the initial value problem  (\ref{E}) (\ref{Ei}) admits a weak solution $u(t,x) $,
in the sense of Definition \ref{def1}, defined for all $(t,x)\in\mathbb{R}^+\times\mathbb{R}$. The solution also satisfies following properties.
\begin{itemize}
\item[(i)] $u(t,x)$ is H\"older continuous with exponent $1-\frac{ 1}{2\lambda}$ on both $t$ and $x$.
\item[(ii)] When $\lambda\geq 2$, energy density $u^2+u_x^2$ is conserved for any time $t\geq 0$, i.e.
\be\label{energy1}
\mathcal{E} (t)=\int_{\mathbb{R}} u^2(t,x)+u_x^2(t,x)dx =\mathcal{E}(0)\quad\hbox{for any}\quad t\geq0;
\ee
When $\lambda=1$ (Camassa-Holm), $\mathcal{E} (t)=\mathcal{E}(0)$ for almost any time $t\geq 0$.

\item[(iii)]  The balance law (\ref{E-3})
is satisfied in the following sense.

There exists a family of Radon measures $\{\mu_{(t)},\, t\geq0\}$,  depending continuously on time and w.r.t the topology of weak convergence of measures.
For every $t\geq 0$,  the absolutely continuous part of $\mu_{(t)}$ w.r.t. Lebesgue measure has density $u_x^{2\lambda}(t,\cdot)$, which provides a
measure-valued solution to the balance law
\be\label{weak_en}
 \int_{0}^\infty \left( \int_{\mathbb R}(\phi_t+u^{\lambda}\phi_x)d\mu_{(t)} +\int_{\mathbb R}(u^{\lambda+1}  - P-Q_x)2\lambda u_x^{2\lambda-1}\phi\, dx\right)dt-\int_{\mathbb R} (u_{0})_{x}^{2\lambda}\phi(0,x)dx=0,
\ee
for every test function $\phi\in C_c^1(\mathbb{R}^+\times \mathbb{R})$. Furthermore, when $\lambda\geq 2$, for almost every $t\geq0$, the singular part   of $\mu_{(t)}$ concentrates on the set where
$u=0$, i.e. when $c'(u)=0$ with wave speed $c(u)=u^\lambda$. When $\lambda=1$, for almost any $t\geq 0$,
 $d \mu_{(t)}=u_x^{2\lambda}(t,x)dx$ since $c'(u)=1$ is always nonzero.

\item[(iv)]
Some continuous dependence result holds. Consider a sequence of initial data ${u_0}_n$ such that
$\|{u_0}_n-u_0\|_{H^1\cap W^{1,2\lambda}}\rightarrow 0$, as $n\rightarrow\infty$. Then
the corresponding solutions $u_n(t,x)$ converge to $u(t,x)$ uniformly for $(t,x)$
in any bounded sets.
\end{itemize}
\end{theorem}

\begin{remark}
When $\lambda=1$ (Camassa-Holm), the higher order and lower order energy laws are the same. As proved in \cite{BC2}, the energy $\mathcal{E}(t)$ might concentrate when $\mathcal{E}(t)<\mathcal{E}(0)$, but this will happen at most
for a measure zero time. On the other hand, there exists a family of radon measures $\{\mu_{(t)}+u^2(t,x)dx\}$ as in part (iii) of Theorem \ref{main}, such that $\{\mu_{(t)}+u^2(t,x)dx\}$ is conserved for any time. 

In fact, because $c(u)=u$ is linear with $c'(u)=1$, for Camassa-Holm equation, wave breaking phenomena happen transiently, or in another word, $\{\mu_{(t)}\}$ is absolute continuous with respect to the Lebesgue measure for almost any time.

For Novikov equation with $\lambda=2$, in \cite{CCL}, the lower order energy $\mathcal{E}(t)$ was proved to be conserved for any time, and meanwhile, a higher order energy was proved to be conserved in the sense of Radon measure. However, since $c(u)$ is nonlinear, the wave breaking phenomenon is not necessarily transient. These properties are also true when $\lambda\geq 2$. Similarly, the wave breaking phenomenon for the variational wave equation is in general not transient \cite{BC}.
\end{remark}

This paper is divided into 6 sections. In Section 2, we provides uniform bounds on $P$, $P_x$, $Q$ and $Q_x$. In Section 3, we derive a semilinear system on new characteristic coordinates, then prove the existence of local and global solutions for this semilinear system in Sections 4 and 5, respectively. Finally, by an inverse transform in Section 6, we prove the existence of conservative solution for the original equation \eqref{E0} and finish the proof of Theorem \ref{main}.

\renewcommand{\theequation}{\thesection.\arabic{equation}}
\setcounter{equation}{0}
\section{Energy and uniform bounds on $P$ $P_x$, $Q$ and $Q_x$}
Starting from now to the end of this paper, we always assume that $\lambda=1$ or $\lambda\geq2$.


  Now, for smooth solutions of (\ref{E}),   multiplying (\ref{E}) by $u$, and multiplying (\ref{E-2}) by $u_x$ then doing integration by part,  we can deduce that
  \[\frac 12\frac{d}{dt} \int u^2 dx+\int(\frac{u^{\lambda+2}}{\lambda+2})_xdx+\int P_xudx+\int Qudx=0,\]
  and
  \[\frac12 \frac{d}{dt}\int  u_x^2 dx+\int\left(\frac 12  u^{\lambda}(u_x^2)_x+\frac 12  u^{\lambda-1}u_x^3 - (\frac{u^{\lambda+2}}{\lambda+2})_x-P_xudx- Q_{xx}u\right)dx=0.
  \]
 Then,   from the properties of $p$, we get $$\int Q_{xx}udx= \int Qudx- \frac {\lambda-1}2 \int u^{\lambda-1}u_x^3 dx$$ and the following lower order energy conservation law
  \be\label{def_E}
  \mathcal{E}(t):=\int_{-\infty}^{\infty} (u^2(t,x)+u_x^2(t,x))dx=\mathcal{E}(0)=:\mathcal{E}_0.\ee
Recall that $P$, $P_x$, $Q$, $Q_x$ are both defined by convolutions in \eqref{PQ}. It is easy to have the following two estimates on $P$ and $P_x$.
  \be\label{PI1}\|P(t)\|_{L^{\infty}},\ \  \|P_x(t)\|_{L^{\infty}}\leq \|\frac 12 e^{-|x|}\|_{L^{\infty}} \|u\|_{L^{\infty}}^{\lambda-1}
 \| (\lambda -\frac12)u_x^2+ u^{2}\|_{L^1}\leq C_1 \mathcal{E}_0^{\frac{\lambda +1}{2}},\ee
\be\label{PI2}\|P(t)\|_{L^{i}},\ \  \|P_x(t)\|_{L^{i}}\leq \|\frac 12 e^{-|x|}\|_{L^{i}} \|u\|_{L^{\infty}}^{\lambda-1}
 \| (\lambda -\frac12)u_x^2+ u^{2}\|_{L^1}\leq C_2 \mathcal{E}_0^{\frac{\lambda +1}{2}},\ee
 for some constants $C_1$ and $C_2$ and any $i\geq 1$ and when $\lambda\geq 1$.

To find a uniform estimate on  $Q$ and $Q_x$ which both include $u_x^3$ in the convolution, we need to first get a uniform estimate on $\| u_x\|_{L^{2\lambda}}$ by the higher order energy balance law \eqref{E-3}. In fact, by \eqref{E-3}, we know that
\[
\begin{split}
\frac{d}{dt}\int |u_x|^{2\lambda} \, dx&=\int(u^{\lambda+1}
 - P)2\lambda u_x^{2\lambda-1}\, dx -\int Q_x 2\lambda u_x^{2\lambda -1}\, dx\\
 &\leq 2\lambda ( \|u\|_{L^{\infty}}^{\lambda+1}+  \|P\|_{L^{\infty}})\, \int |u_x|^{2\lambda -1}\, dx
 +C_0\|u\|_{L^{\infty}}^{\lambda-2}\int |u_x|^3\, dx\, \int |u_x|^{2\lambda -1}\, dx,
\end{split}
\]
for some constant $C_0$. Notice that when $\lambda\geq 2$
\be
(\int |u_x|^3\, dx)\, (\int |u_x|^{2\lambda -1}\, dx)\leq (\int |u_x|^2\, dx)\, (\int |u_x|^{2\lambda }\, dx)
\ee
which can be proved using the following two H\"older inequalities:
\be\label{i2}
\int |u_x|^3\, dx\leq (\int |u_x|^2\, dx)^{1- \frac{1}{2\lambda-2}}\, (\int |u_x|^{2\lambda}\, dx)^{\frac{1}{2\lambda-2}},
\ee
and
\be\label{i3}
\int |u_x|^{2\lambda -1}\, dx\leq (\int |u_x|^2\, dx)^{\frac{1}{2\lambda -2}}\, (\int |u_x|^{2\lambda}\, dx)^{1- \frac{1}{2\lambda -2}}.
\ee
So we have
\[
\frac{d}{dt}\int |u_x|^{2\lambda}\, dx\leq C'(E_0)(\int |u_x|^{2\lambda}\, dx+C''(E_0))
\]
for some constants $C'(E_0)$ and $C''(E_0)$. This tells that for any time $t\in[0,\T]$,
\be\label{uxk}
(\int |u_x|^{2\lambda}\, dx)(t)\leq e^{C'(E_0) \T}(\int |u_x|^{2\lambda}\, dx)(0)+ C''(E_0)(e^{C'(E_0) \T}-1)=:J_0(\T).
\ee
When $\lambda\geq 2$, using \eqref{uxk} and \eqref{def_E}, one can bound $(\int |u_x|^3\, dx)(t)$ by a constant depending on $\T$. When $\lambda=0,$ $Q=0$. So, for any given time $\T>0$, we can easily show the following time dependent bounds, when $t\in[0,\T]$:
 \be\label{QI1}\|Q(t)\|_{L^{\infty}},\ \  \|Q_x(t)\|_{L^{\infty}}\leq C_3(\mathcal{E}_0,J_0(\T)),\ee
\be\label{QI2}\|Q(t)\|_{L^{i}},\ \  \|Q_x(t)\|_{L^{i}}\leq C_4(\mathcal{E}_0,J_0(\T)),\ee
for any $i>1$. These inequalities are also right when $\lambda=1$, sine $Q=0$ when $\lambda=1$.

We will prove similar bounds as \eqref{PI1}-\eqref{PI2} and \eqref{QI1}-\eqref{QI2} for energy conservative weak solutions in Section 5.

%
\section{The semi-linear system}
As mentioned earlier, solutions of \eqref{E0} or equivalently quasi-linear equation (\ref{E}) will form cusp singularity generically. We will consider weak solutions in the energy space $H^1(\mathbb{R})\cap W^{1,2\lambda}(\mathbb{R})$. However, it is challenge to solve (\ref{E}) directly.

Instead, we will first derive a semi-linear system \eqref{S-1} in this section, by studying smooth solutions of (\ref{E}). Next, we will find the solution of this semi-linear system in Sections 4-5 and then transform it back to a weak solution of  (\ref{E}) in Section 6.  This idea was first used in \cite{BC2} for Camassa-Holm equation with $\lambda = 1$ and in \cite{CS2015} for a class of unitary direction nonlinear wave equation with general $\lambda$.

Now we derive the semilinear system by studying smooth solutions. First,
the equation of characteristic is
 \be\label{C}\frac{dx(t)}{dt}=u^{\lambda}(t,x(t)).\ee
 We denote the characteristic passing through the point $(t,x)$ as $\tau\mapsto x(\tau;t,x)$, for any time $\tau\geq0$.
Now we introduce a new coordinates $(T,Y)$ defined by
\be\label{2.1}(t,x)\mapsto (T,Y),\ \ \ {\rm where }\ \ \ T=t, \ \ Y=\int_0^{x(0;t,x)} (1+(u_0)_x^2)^{\lambda}dx.\ee
So, $Y=Y(t,x)$ is  a characteristic coordinate satisfying
 \be\label{C2}Y_t+u^{\lambda} Y_x=0.\ee
Then for any smooth function $f:=f(T,Y(t,x))$, it is easy to get
 \[\left\{\begin{array}{lll}&f_T=f_T(\ T_t+u^{\lambda}T_x)+f_Y\ (Y_t+u^{\lambda}  Y_x)=f_t+u^{\lambda}f_x,\\[2mm]
&f_x=f_T\ T_x+ f_Y\ Y_x=f_Y\ Y_x.\end{array}\right.\]
We introduce the following new variables:
 \be\label{T} v:=2\arctan u_x   \ \ \ {\rm and }\ \ \ \xi:=\frac{(1+u_x^2)^{\lambda}}{Y_x}.\ee
It is easy to get
 \be\label{T1}\tan \frac v2= u_x,\ \ \ \ \ \ \ \ 1+u_x^2=\sec^2\frac v2,\ \ \ \ \ \frac{1}{1+u_x^2}=\cos^2\frac v2,\ee
\be\label{T2}\frac{u_x^2}{1+u_x^2}=\sin^2\frac v2,\ \  \frac{u_x}{1+u_x^2}=\frac12\sin v,\ \ \ x_Y=\frac{\xi}{(1+u_x^2)^{\lambda}}=\cos^{2\lambda}\frac v2 \cdot \xi.\ee
Then, by (\ref{E}), we have
\[u_T=u_t+u^{\lambda}u_x=-P_x-Q.\]
From (\ref{E-2}), we deduce that
\[\begin{split}
v_T&=\frac{2}{1+u_x^2} (u_{xt}+u^{\lambda}u_{xx})\\
&=-u^{\lambda-1}\frac{u_x^2}{1+u_x^2}+2u^{\lambda+1}\frac{1}{1+u_x^2}-\frac{2}{1+u_x^2}(P+Q_x)\\
&=-u^{\lambda-1}\sin^2\frac v2+2u^{\lambda+1}\cos^2\frac v2-2\cos^2\frac v2 (P +Q_x ).
\end{split}\]
Finally, by (\ref{E-3}), we deduce that
\[\begin{split}
\xi_T&=\frac{1}{Y_x}\lambda (1+u_x^2)^{\lambda-1}(u_x^2)_T-\frac{(1+u_x^2)^{\lambda}}{Y_x^2}(Y_x)_T\\
&=\frac{\lambda}{Y_x} (1+u_x^2)^{\lambda-1}[(u_x^2)_t+u^{\lambda}(u_x^2)_x+\frac 1\lambda (1+u_x^2)\lambda u^{\lambda-1}u_x]\\
&=\frac{(1+u_x^2)^{\lambda}}{Y_x}\frac{\lambda u^{\lambda-1} u_x}{1+u_x^2}+\frac{(1+u_x^2)^{\lambda}}{Y_x}\frac {u_x}{ 1+u_x^{2}}\frac {1}{ u_x^{2\lambda-1}}[(u_x^{2\lambda})_t+(u^{\lambda}u_x^{2\lambda})_x]\\
&=\frac{\lambda}{2}u^{\lambda-1} \xi \sin v
+\lambda u^{\lambda+1}\xi \sin v -\lambda  \xi \sin v( P+Q_x ).
\end{split}\]
In conclusion, we derive a new semi-linear system,
\be\label{S-1}
\left\{\begin{array}{lll}
&u_T=-P_x-Q,\\[2mm]
& v_T=-u^{\lambda-1}\sin^2\frac v2+2u^{\lambda+1}\cos^2\frac v2-2\cos^2\frac v2(P  +Q_x),\\[2mm]
& \xi_T=\frac{\lambda}{2}u^{\lambda-1} \xi \sin v
+\lambda u^{\lambda+1}\xi \sin v -\lambda  \xi \sin v( P  +Q_x),
\end{array}\right.
\ee
with initial data
\be\label{S-2}
\left\{\begin{array}{lll}
&u(0,Y)=u_0(x_0(Y)),\\[2mm]
& v(0,Y)=2 \arctan (u_0)_x(x_0(Y)),\\[2mm]
& \xi(0,Y)=1.
\end{array}\right.
\ee
Next, we find the expressions of $P$, $P_x$, $Q$ and $Q_x$ under the new coordinates $(T,Y)$. It follows from the last formula in (\ref{T2}) that
\[x(T,Y)-x(T,Y^{'})=\int_{Y^{'}}^Y \cos^{2\lambda}\frac{v(T,s)}{2} \cdot \xi(T,s)ds.\]
We take $x=x(T,Y^{'})$, then $dx=\frac{\xi}{(1+u_x^2)^{\lambda}}d Y^{'}$.
Thus, we get
\be\label{P1}\begin{split}P(T,Y)
&=\frac 12\int_{-\infty}^{+\infty}e^{-|\int_{Y}^{Y^{'}}\cos^{2\lambda}\frac{v(s)}{2} \cdot\xi(s)ds|}
\\
&\qquad\qquad\cdot\ [(\lambda-\frac12) u^{\lambda-1}\sin^2\frac{v(Y^{'})}{2} \cos^{2\lambda-2}\frac{v(Y^{'})}{2} + u^{\lambda+1}\cos^{2\lambda}\frac{v(Y^{'})}{2}]\cdot \xi(Y^{'}) dY^{'},
\end{split}\ee

\be\label{P2}\begin{split}P_x(T,Y)
&=\frac 12(\int_{Y}^{+\infty}-\int_{-\infty}^{Y})e^{-|\int_{Y}^{Y^{'}}\cos^{2\lambda}\frac{v(s)}{2} \cdot\xi(s)ds|}\\
&\qquad\qquad\cdot [(\lambda-\frac12) u^{\lambda-1}\sin^2\frac{v(Y^{'})}{2} \cos^{2\lambda-2}\frac{v(Y^{'})}{2}+ u^{\lambda+1}\cos^{2\lambda}\frac{v(Y^{'})}{2}]\cdot \xi(Y^{'}) dY^{'},
\end{split}\ee

\be\label{Q1}\begin{split}Q(T,Y)
&=\frac{\lambda -1}{4} \int_{-\infty}^{+\infty}e^{-|\int_{Y}^{Y^{'}}\cos^{2\lambda} \frac{v(s)}{2} \cdot\xi(s)ds|} [ u^{\lambda-2}\sin^3\frac{v(Y^{'})}{2} \cos^{2\lambda-3}\frac{v(Y^{'})}{2}] \cdot \xi(Y^{'}) dY^{'},\end{split}\ee
and
\be\label{Q2}Q_x(T,Y)
=\frac{\lambda -1}{4} (\int_{Y}^{+\infty}-\int_{-\infty}^{Y})e^{-|\int_{Y}^{Y^{'}}\cos^{2\lambda} \frac{v(s)}{2} \cdot\xi(s)ds|}   \cdot[ u^{\lambda-2}\sin^3\frac{v(Y^{'})}{2} \cos^{{2\lambda}-3}\frac{v(Y^{'})}{2}] \xi(Y^{'}) dY^{'}. \ee

To be more precise when $\lambda$ is not an integer, $\cos^{2\lambda}\frac v2$ and $ \cos^{2\lambda-3} \frac v2$  mean $(\cos^2\frac v2)^\lambda$ and $ (\cos^2x)^\lambda \cos^{-3} \frac v2$, respectively, and similar for other terms. 

\renewcommand{\theequation}{\thesection.\arabic{equation}}
\setcounter{equation}{0}
\section{Local existence of semi-linear system \eqref{S-1}-\eqref{S-2}}
In this section, we will prove the following local existence result of weak solutions of semi-linear system (\ref{S-1})-(\ref{S-2}) using the contraction mapping theory.
\begin{proposition}\label{prop_4.1} Let $\lambda=1$ or $\lambda\geq2$. If the initial data
$u_0\in H^1(\mathbb{R}) \cap W^{1,2\lambda}(\mathbb{R})$, then there exists $\T_0>0$, such that, the Cauchy problem (\ref{S-1})-(\ref{S-2}) has a unique solution on the interval $[0, \T_0]$.

\end{proposition}
\begin{proof}
We introduce a space $X$ defined as
\[X:= H^1(\mathbb{R})\cap W^{1,2\lambda}(\mathbb{R})\times[L^2(\mathbb{R})\cap L^{\infty}(\mathbb{R})]\times L^{\infty}(\mathbb{R})\]
with its norm $\|(u,v,\xi)\|_{X}=\|u\|_{H^1\cap W^{1,2\lambda}}+\|v\|_{L^2}+\|v\|_{L^{\infty}}+\|\xi\|_{L^{\infty}}$.  We will prove the existence of a fixed point of the map $\Phi(u,v,\xi)=(\widetilde{u},\widetilde{v},\widetilde{\xi})$, where $(\widetilde{u},\widetilde{v},\widetilde{\xi})$ is defined by
\[
\left\{\begin{array}{lll}
& \widetilde{u}(T,Y)=u_0(x_0(Y)-\int_0^{T}(P_x(\tau,Y)+Q(\tau,Y))d\tau,\\[4mm]
& \widetilde{v}(T,Y)=2 \arctan (u_0)_x(x_0(Y))+\int_0^{T}[-u^{\lambda-1}\sin^2\frac v2+2\cos^2\frac v2(u^{\lambda+1}-P-Q_x)d\tau,\\[4mm]
&  \widetilde{\xi}(T,Y)=1+\int_0^{T}[\frac{\lambda}{2}u^{\lambda-1} \xi \sin v
+ \lambda u^{\lambda+1}\xi \sin v -\lambda  \xi \sin v( P  +Q_x)]d\tau,
\end{array}\right.
\]
which directly comes from  \eqref{S-1}-\eqref{S-2}.
Here $P$, $P_x$, $Q$ and $Q_x$ are defined by (\ref{P1})-(\ref{Q2}).

As the standard ODE theory in the Banach space, we only have to prove that all functions on the right hand side of (\ref{S-1}) are locally Lipschitz continuous with respect to $(u,v,\xi)$ in $X$.
More precisely, let $K\subset X$ be a bounded domain defined by
\[K=\{(u,v,\xi)|  \|u\|_{W^{1,2\lambda}}\leq A, \|v\|_{L^2}\leq B,  \|v\|_{L^{\infty}}\leq \frac{3\pi}{2},   \xi(Y)\in[C^-, C^+]   \}\]
  for  a.e. $Y\in \mathbb{R}$, where    $A$, $B$, $C^-$ and  $C^+$  are positive constants. Then, if  the map $\Phi(u,v,\xi)$ is Lipschitz continuous on $K$, then
$\Phi(u,v,\xi)$ has a fixed point by the  contraction mapping theory, which means the existence of local solutions for \eqref{S-1}-\eqref{S-2}.

First, it follows from the Sobolev embedding that
$\|u\|_{L^{\infty}}\leq C\|u\|_{H^1}$.
Moreover, due to the uniform bounds of $v$ and $\xi$ in $K$, the maps
\[   u^{\lambda-1}\sin^2\frac v2, \ u^{\lambda+1}\cos^2\frac v2, \   \ u^{\lambda-1} \xi\ \sin v,  \]
are all Lipschitz continuous as maps from $K$ into $L^2(\mathbb{R})$, from $K$ into $L^{2\lambda}(\mathbb{R})$, as well as  from $K$ into $L^{\infty}(\mathbb{R})$.

Secondly,
in order to estimate the singular integrals in $P$, $P_x$, $Q$ and $Q_x$, we observe that $e^{-|\int_{Y}^{Y^{'}}\cos^{2\lambda} \frac{v(s)}{2} \cdot\xi(s)ds|}\leq 1$. But this is not enough to control the above singular integrals. We will use the following important lemma, where the proof can be found in the Appendix A.

\begin{lemma}\label{le_4.2}
  Let $f\in L^q(\mathbb{R})$ with $q\geq 1$. If $\|v\|_{L^2}\leq B$, then for any $(u,v,\xi)\in K$,
\be\label{Singular} \|\int_{-\infty}^{+\infty}  e^{-|\int_{Y}^{Y^{'}}\cos^{2\lambda}\frac{v(s)}{2} \cdot\xi(s)ds|} f(Y^{'})d Y^{'}\|_{L^q}\leq  \|g\|_{L^1} \ \|f\|_{L^q},\ee
where $g(z)=\min\{1,e^{ \frac{C^-}{2^{\lambda}}(\frac{9 }{2 }B^2-|z|]}\}$ and $\|g\|_{L^1}=9B^2+\frac{2^{\lambda+1}}{ C^-}$.
\end{lemma}

Thirdly, by (\ref{Singular}), we give a priori bounds on $P$, $P_x$, $Q$ and $Q_x$, which implies $P$, $P_x$, $Q$ and $Q_x$ $\in W^{1,2\lambda}$. More precisely, we deduce that
 $P$, $\partial_YP$, $P_x$, $\partial_YP_x$, $Q$ $\partial_YQ$, $Q_x$, $\partial_YQ_x$ $\in L^{2\lambda}$.
We will only prove the bounds on $Q$, and the estimates for $P$ can be obtained similarly. Using (\ref{Singular}), we derive several inequalities.

\be\label{Q-1}\begin{split}\|Q\|_{L^{2\lambda}}&\leq \frac{{(\lambda-1)} C^+}{4} \|g\|_{L^1}\|u^{\lambda-2}\sin^3\frac{v(Y^{'})}{2} \cos^{{2\lambda}-3}\frac{v(Y^{'})}{2}\|_{L^{2\lambda}}\\
&\leq \frac{{(\lambda-1)} C^+}{4} \|g\|_{L^1}\| u\|_{L^{\infty}}^{\lambda-2}  \| \sin^{3-{\frac{1}{\lambda}}}\frac{v(Y^{'})}{2} \cos^{{2\lambda}-3}\frac{v(Y^{'})}{2}\|_{L^{\infty}}           \|v^{{\frac{1}{\lambda}}}\|_{L^{2\lambda}}  \\
&\leq  \frac{{(\lambda-1)} C^+}{4} \|g\|_{L^1}\| u\|_{L^{\infty}}^{\lambda-2}\|v\|_{L^2}^{{\frac{1}{\lambda}}};\end{split}\ee

\be\label{Q-2}\begin{split}\partial_YQ(Y)&=
\frac{{(\lambda-1)} }{4} \int_{-\infty}^{+\infty}e^{-|\int_{Y}^{Y^{'}}\cos^{2\lambda}\frac{v(s)}{2} \cdot\xi(s)ds|} \cdot[\cos^{2\lambda}\frac{v(Y)}{2} \cdot\xi(Y)]\ {\rm sign} (Y^{'}-Y)\\
&\qquad \qquad\qquad \cdot[ u^{\lambda-2}\sin^3\frac{v(Y^{'})}{2} \cos^{{2\lambda}-3}\frac{v(Y^{'})}{2}] \cdot \xi(Y^{'}) dY^{'};\end{split}\ee

\be\label{Q-3} \|\partial_YQ\|_{L^{2\lambda}}\leq \frac{{(\lambda-1)} (C^+)^2}{4} \|g\|_{L^1}\| u\|_{L^{\infty}}^{\lambda-2}\|v\|_{L^2}^{{\frac{1}{\lambda}}},\ee
which gives that
\be\label{Q-4}\begin{split}\|Q_x\|_{L^{2\lambda}}\leq \frac{{(\lambda-1)} C^+}{2} \|g\|_{L^1}\|u^{\lambda-2}\sin^3\frac{v(Y^{'})}{2} \cos^{{2\lambda}-3}\frac{v(Y^{'})}{2}\|_{L^{2\lambda}}\leq  \frac{{(\lambda-1)} C^+}{2} \|g\|_{L^1}\| u\|_{L^{\infty}}^{\lambda-2}\|v\|_{L^2}^{{\frac{1}{\lambda}}};\end{split}\ee
and
\be\label{Q-5}\begin{split}\partial_YQ_x(Y)&=-\frac{{(\lambda-1)} }{2}  u^{\lambda-2}\sin^3\frac{v(Y)}{2} \cos^{{2\lambda}-3}\frac{v(Y)}{2} \cdot \xi(Y)\\
&\quad+ \frac{{(\lambda-1)} }{4} (\int_{Y}^{+\infty}-\int_{-\infty}^{Y})e^{-|\int_{Y}^{Y^{'}}\cos^{2\lambda}\frac{v(s)}{2} \cdot\xi(s)ds|}\cdot[\cos^{2\lambda}\frac{v(Y^{'})}{2} \cdot\xi(Y^{'})]\\
&\quad\quad \cdot{\rm sign} (Y^{'}-Y) \cdot[ u^{\lambda-2}\sin^3\frac{v(Y^{'})}{2} \cos^{{2\lambda}-3}\frac{v(Y^{'})}{2}] \cdot \xi(Y^{'}) dY^{'},\end{split}\ee
which gives that
\be\label{Q-6}\begin{split}\|\partial_YQ_x\|_{L^{2\lambda}}\leq\frac{{(\lambda-1)} C^+}{2}  \| u\|_{L^{\infty}}^{\lambda-2}\|v\|_{L^2}^{{\frac{1}{\lambda}}}+ \frac{{(\lambda-1)} (C^+)^2}{4} \|g\|_{L^1}\| u\|_{L^{\infty}}^{\lambda-2}\|v\|_{L^2}^{{\frac{1}{\lambda}}}.\end{split}\ee
Hence, it is easy to show that  $\Phi (u,v,\xi)$ is from $K$ to $K$, when $\T_0$ is small enough.

Finally,   we prove the map $\Phi (u,v,\xi)$ is Lipschitz continuous with respect to $(u,v,\xi)$ in $K$. More precisely, we need to prove that the following partial derivatives
\be\label{L1}\frac{\partial P}{\partial u}, \ \frac{\partial P}{\partial v}, \ \frac{\partial P}{\partial \xi}, \ \frac{\partial P_x}{\partial u},\ \frac{\partial P_x}{\partial v},\ \frac{\partial P_x}{\partial \xi}, \frac{\partial Q}{\partial u}, \ \frac{\partial Q}{\partial v}, \ \frac{\partial Q}{\partial \xi}, \ \frac{\partial Q_x}{\partial u},\ \frac{\partial Q_x}{\partial v},\ \frac{\partial Q_x}{\partial \xi}\ee
are uniformly bounded for any $(u,v,\xi)\in K$.  For sake of illustration, we only give the detail estimates for terms including $Q$. Then the estimates for other terms can be obtained similarly.

More precisely,
we will prove that $ \frac{\partial Q}{\partial u}$ and $\frac{\partial Q_x}{\partial u}$ are bounded linear operators from $W^{1,{2\lambda}}(\mathbb{R})$ to $W^{1,{2\lambda}}(\mathbb{R})$;  $\frac{\partial Q}{\partial v}$ and $\frac{\partial Q_x}{\partial v}$ are  bounded linear operators from $L^2(\mathbb{R})\cap L^{\infty}(\mathbb{R})$ to $W^{1,{2\lambda}}(\mathbb{R})$; $\frac{\partial Q}{\partial \xi}$ and $\frac{\partial Q_x}{\partial \xi}$ are  bounded linear operators from $  L^{\infty}(\mathbb{R})$ to $W^{1,{2\lambda}}(\mathbb{R})$. The linearity of these maps are trivial.

To show that these maps are bounded, we will prove that, for any $(u,v,\xi)\in K$, the derivatives \[ \frac{\partial Q}{\partial u}: W^{1,{2\lambda}}(\mathbb{R}) \mapsto L^{2\lambda}(\mathbb{R}),\ \ \    \frac{\partial (\partial_YQ)}{\partial u}: W^{1,{2\lambda}}(\mathbb{R}) \mapsto L^{2\lambda}(\mathbb{R}),\]  \[ \frac{\partial Q}{\partial v}: L^2(\mathbb{R})\cap L^{\infty}(\mathbb{R}) \mapsto L^{2\lambda}(\mathbb{R}), \ \ \ \frac{\partial (\partial_YQ)}{\partial v}: L^2(\mathbb{R})\cap L^{\infty}(\mathbb{R}) \mapsto L^{2\lambda}(\mathbb{R}),\]  \[\frac{\partial Q}{\partial \xi}:  L^{\infty}(\mathbb{R}) \mapsto L^{2\lambda}(\mathbb{R}), \ \ \  \frac{\partial (\partial_YQ)}{\partial \xi}:   L^{\infty}(\mathbb{R}) \mapsto L^{2\lambda}(\mathbb{R})\] are bounded. We will only take $ \frac{\partial Q}{\partial u}$ and $ \frac{\partial (\partial_YQ)}{\partial u}$ as examples. One can treat the other cases using very similar method. Indeed,
$ \frac{\partial Q}{\partial u}$ is a linear operator defined as following. For any $\varphi \in W^{1,{2\lambda}}(\mathbb{R})$,
 \[[\frac{\partial Q}{\partial u}\cdot \varphi](Y)=\frac{{(\lambda-1)} }{4} \int_{-\infty}^{+\infty}e^{-|\int_{Y}^{Y^{'}}\cos^{2\lambda}\frac{v(s)}{2} \cdot\xi(s)ds|} [ \sin^3\frac{v(Y^{'})}{2} \cos^{{2\lambda}-3}\frac{v(Y^{'})}{2}] \xi(Y^{'}) (\lambda-2)u^{{\lambda-3}} \cdot \varphi dY^{'} \]
 Then,
 \[ \|[\frac{\partial Q}{\partial u}\cdot \varphi](Y)\|_{L^{2\lambda}}\leq \frac{{(\lambda-1)}(\lambda-2) C^+}{4} \|g\|_{L^1}\| u\|_{L^{\infty}}^{{\lambda-3}}\|v\|_{L^2}^{{\frac{1}{\lambda}}}\cdot \|\varphi\|_{L^{\infty}}.\]
 From the Sobolev embedding, we see that $\|\varphi\|_{L^{\infty}}\leq \|\varphi\|_{W^{1,{2\lambda}}}$, and
 \[\|\frac{\partial Q}{\partial u} \|\leq \frac{{(\lambda-1)}(\lambda-2) C^+}{4} \|g\|_{L^1}\| u\|_{L^{\infty}}^{{\lambda-3}}\|v\|_{L^2}^{{\frac{1}{\lambda}}}.\]
 $ \frac{\partial (\partial_YQ)}{\partial u}$ is the linear operator defined by $\forall\ \varphi \in W^{1,{2\lambda}}(\mathbb{R})$,
 \[\begin{split}[\frac{\partial (\partial_YQ)}{\partial u}\cdot \varphi](Y)
=&\frac{{(\lambda-1)} }{4} \int_{-\infty}^{+\infty}e^{-|\int_{Y}^{Y^{'}}\cos^{2\lambda}\frac{v(s)}{2} \cdot\xi(s)ds|} \cdot[\cos^{2\lambda}\frac{v(Y^{'})}{2} \cdot\xi(Y^{'})]\ {\rm sign} (Y^{'}-Y)\\
&\qquad \qquad  \cdot[ \sin^3\frac{v(Y^{'})}{2} \cos^{{2\lambda}-3}\frac{v(Y^{'})}{2}]   \xi(Y^{'})   (\lambda-2)u^{{\lambda-3}} \cdot \varphi dY^{'} \end{split}\]
 Then,
 \[ \|[\frac{\partial (\partial_YQ)}{\partial u}\cdot \varphi](Y)\|_{L^{2\lambda}}\leq \frac{{(\lambda-1)}(\lambda-2) (C^+)^2}{4} \|g\|_{L^1}\| u\|_{L^{\infty}}^{{\lambda-3}}\|v\|_{L^2}^{{\frac{1}{\lambda}}}\cdot \|\varphi\|_{L^{\infty}}.\]
 From the Sobolev embedding inequality, we see that $\|\varphi\|_{L^{\infty}}\leq C \|\varphi\|_{W^{1,{2\lambda}}}$, and
 \[\|\frac{\partial (\partial_YQ)}{\partial u} \|\leq \frac{{(\lambda-1)}(\lambda-2) (C^+)^2}{4} \|g\|_{L^1}\| u\|_{L^{\infty}}^{{\lambda-3}}\|v\|_{L^2}^{{\frac{1}{\lambda}}}.\]

Now, applying the standard ODE local existence theory in Banach space, we complete the proof of Proposition \ref{prop_4.1}.
\end{proof}

\renewcommand{\theequation}{\thesection.\arabic{equation}}
\setcounter{equation}{0}
\section{Global existence of semi-linear system \eqref{S-1}-\eqref{S-2}}

In this section, we will  prove that the local solution for the Cauchy problem (\ref{S-1})-(\ref{S-2}) can be extended  to a global one.
We will first prove that there exists an a positive priori bound $C(E_0, \T)$ such that
\be\label{B}\|u(T)\|_{W^{1,{2\lambda}}}+\|v(T)\|_{L^2}+\|v(T)\|_{L^{\infty}}+\|\xi(T)\|_{L^{\infty}}+\|\frac{1}{\xi(T)}\|_{L^{\infty}}\leq C(E_0,J_0(\T)),\ee
for any $\T\geq 0$. Using this a priori bound, we can extend the local solution to $T\in[0,\T]$. Then as $\T$ can be arbitrarily large, we can uniquely extend the solution to any time $T$, so we prove the global existence.

\begin{proposition}\label{pro_5.1} Let $\lambda=1$ or $\lambda\geq2$. If the initial data
$u_0\in W^{1,{2\lambda}}(\mathbb{R})$, then the Cauchy problem (\ref{S-1})-(\ref{S-2}) has a unique solution, defined for all times $T\in \mathbb{R}^+$.
\end{proposition}
\begin{proof} In Proposition \ref{prop_4.1}, we have proved the existence of local solution for the Cauchy problem (\ref{S-1})-(\ref{S-2}).  The key part of proof for this proposition is to get the a priori estimate (\ref{B}) for any given $\T$.

From now on, we always assume that  $T\in[0,\T]$. In fact, after showing the existence of unique solution in $T\in[0,\T]$ for any $\T>0$, one can uniquely extend the solution to any time $T$, then the proof of this proposition is done.

\smallskip

\paragraph{\bf 1.}
First, we prove the energy conservation and balance laws under the new coordinate $(T,Y)$.  First of all, when $t=0$, from the transformation (\ref{T})-(\ref{T2}), we see that  $Y_x(0)=(1+(u_0)_x^2)^{{\lambda}}$, then
\[u_Y(0,Y)=\frac{(u_0)_x}{(1+(u_0)_x^2)^{{\lambda}}}=\frac12  \sin v(0,Y) \cos^{2{(\lambda-1)}}\frac{v(0,Y) }{2}\ \ {\rm and }\ \ \xi(0,Y)=1.\] We claim that  for all times $T$,
\be\label{uY}u_Y=\frac12 \xi \,\sin v \,\cos^{2{(\lambda-1)}}\frac{v}{2}.\ee
Indeed, from the first identity in (\ref{S-1}), for all times $T$, we have
\[\begin{split}(u_Y)_T&=(u_T)_Y=-(P_x)_Y-Q_Y\\
&=-\frac{P_{xx}+Q_x}{Y_x}\\
&=\xi [(\lambda - \frac 12)u^{{(\lambda-1)}}\sin^2\frac{v}{2}\cos^{2{(\lambda-1)}}\frac v2+u^{\lambda+1}\cos^{2\lambda}\frac v2-\cos^{2\lambda}\frac v2(P+Q_x)].\end{split}\]
It follows from the last two identity in (\ref{S-1}) that for all  times  $T$
\[\begin{split}
&(\frac 12 \xi \sin v \cos^{2{(\lambda-1)}}\frac v2 )_T
=(u_Y)_T.
\end{split}\]
Using the fact  $\xi(0,Y)=1$, we know that (\ref{uY}) is always true.  Next, we compute the  $T-$derivative of the energy $E(T)$, where
\be\label{Energy}E(T)=\int_{\mathbb{R}}(u^2\cos^{2 \lambda } \frac {v}{2}+\sin^2\frac {v}{2}\cos^{2{(\lambda-1)}}\frac {v}{2})\ \xi dY.\ee
It is easy to get that
\[\begin{split}
&\frac{d}{dT}\int_{\mathbb{R}}(u^2\cos^{2 \lambda } \frac {v}{2}+\sin^2\frac {v}{2}\cos^{2{(\lambda-1)}}\frac {v}{2})\ \xi dY\\
=&\int_{\mathbb{R}}\{\frac{ \lambda+2 }{2} u^{\lambda+1}\xi \sin v\cos^{2{(\lambda-1)}} \frac v2-\xi \sin v \cos^{2{(\lambda-1)}} \frac v2 (P+Q_x)\\
&-2u\xi \cos^{2 \lambda } \frac v2 (P_x+Q)+\frac{{(\lambda-1)}}{2}\sin^2 \frac v2 \xi\sin v \cos^{2{\lambda-4}}\frac v2\}dY.
\end{split}\]
By (\ref{uY}), we know that
\[  (u^{ \lambda+2 })_Y=\frac{ \lambda+2 }{2} u^{\lambda+1}\xi \sin v\cos^{2{(\lambda-1)}} \frac v2,\]
Then also using
\[P_Y=P_x\, \xi \cos^{2 \lambda } \frac v2,\]
and
\[(Q_x)_Y=-\frac {{(\lambda-1)}}{2}u^{\lambda-2}\xi \sin^3 \frac v2 \cos^{2\lambda-3} \frac v2+Q\xi \cos^{2 \lambda } \frac v2,\]
we know that
\[2[u(P+Q_x)]_Y=\xi\sin v \cos^{2{(\lambda-1)}}\frac v2(P+Q_x)+2u\xi \cos^{2 \lambda } \frac v2 (P_x+Q)-{(\lambda-1)} u^{{(\lambda-1)}}\xi \sin^3 \frac v2 \cos^{2\lambda-3} \frac v2.\]
So
\[\frac{dE(T)}{dT}=\int_{\mathbb{R}} [u^{ \lambda+2 }-2u(P+Q_x)]_YdY=0.
\]
Therefore, we have proved that the lower order energy conservation law
\be\label{Energy2}E(T)=E(0):=E_0.\ee

Meanwhile, we can derive the higher order energy balance law of $\int \xi \sin^{2\lambda}\frac{v}{2} dY$, which is essentially $\int u_x^{2\lambda} dx$. This balance law under $(T,Y)$ coordinates comes from  \eqref{E-3}. Here
\be\label{19_2}
\begin{split}
&\frac{d}{dT}\int \xi \sin^{2\lambda}\frac{v}{2}\, dY\\
=&\int \xi_T\,\sin^{2\lambda}\frac{v}{2}\, dY+\int \xi{\lambda}v_T\sin^{{2\lambda}-1}\frac{v}{2}\cos\frac{v}{2}\,dY\\
=&\int \left(u^{\lambda+1}-P-Q_x\right)\, {2\lambda}\xi \sin^{{2\lambda}-1}\frac{v}{2}\,
\cos\frac{v}{2}\, d Y.
\end{split}
\ee
\smallskip

\paragraph{\bf 2.}
Then we come to find the $L^\infty$ bounds $\|u(T)\|_{L^{\infty}}$, $\|\xi(T)\|_{L^{\infty}}$, $\|\frac{1}{\xi(T)}\|_{L^{\infty}}$ and  $\|v(T)\|_{L^{\infty}}$. It follows from (\ref{uY}) and (\ref{Energy2}) that
\[\begin{split}\sup\limits_{Y\in \mathbb{R}} |u^2(T,Y)|
&\leq \int_{\mathbb{R}}|(u^2)_Y|dY=\int_{\mathbb{R}} |u \sin v \cos^{2{(\lambda-1)}} \frac v2|\ \xi dY\\
&\leq  2\int_{\mathbb{R}} |u\cos^{\lambda} \frac v2 \sin \frac{v}{2} \cos^{{(\lambda-1)}} \frac v2|\ \xi dY\\
&\leq \int_{\mathbb{R}}(u^2\cos^{2 \lambda } \frac {v}{2}+\sin^2\frac {v}{2}\cos^{2{(\lambda-1)}}\frac {v}{2})\ \xi dY,
\end{split}\]
which implies that $\|u(T)\|_{L^{\infty}}$ has a uniform priori bound $E_0^{\frac 12 }$. More precisely, we have the following estimates:
\be\label{Energy3}\left\{\begin{array}{lll}&\|u(T)\|_{L^{\infty}}\leq E_0^{\frac 12 },\\[4mm]
 &  \|\xi  u^2(T)\cos^{2 \lambda } \frac {v(T)}{2}\|_{L^1}\leq E_0,\\[4mm]
 &  \|\xi \sin^2\frac {v(T)}{2}\cos^{2{(\lambda-1)}}\frac {v(T)}{2} \|_{L^1}\leq E_0.\end{array}\right.\ee

To find a bound on $\xi$, we need to first find bounds on $P$ and $Q_x$.
By the definition of $P$,  $P_x$, $Q$ and $Q_x$ in (\ref{P1})-(\ref{Q2}), we deduce that
\be\label{Energy4}\begin{split} \|P(T)\|_{L^{\infty}}, \|P_x(T)\|_{L^{\infty}}&\leq \| \frac 12 e^{-|x|}\|_{L^{\infty}}(\|u(T)\|_{L^{\infty}}^{{(\lambda-1)}}\|\xi  u^2(T)\cos^{2 \lambda } \frac {v(T)}{2}\|_{L^1}\\
&\qquad +(\lambda - \frac 12)\|u(T)\|_{L^{\infty}}^{{(\lambda-1)}} \|\xi \sin^2\frac {v(T)}{2}\cos^{2{(\lambda-1)}}\frac {v(T)}{2} \|_{L^1})\\
&\leq \frac{2 \lambda+1 }{4} E_0^{\frac{\lambda+1}{2}},
\end{split}\ee
which is exactly the same bound found in \eqref{PI1}.

Similarly we can verify that bounds \eqref{QI1} and \eqref{QI2} of $Q$ and $Q_x$ in Section 2 for smooth solutions still hold for any solutions of \eqref{S-1}-\eqref{S-2}. The only difference in the proof is that now we need to use  \eqref{19_2} which comes from the balance law \eqref{E-3} to show that $\int \xi \sin^{2\lambda}\frac{v}{2} dY$, which is essentially $\int u_x^{2\lambda} dx$, has a upper bound depending on $\T$ i.e.
\be\label{j0t}
\int u_x^{2\lambda}\, dx\leq J_0(\T),
\ee
for some constant $J_0(\T)$.

More precisely, inequalities \eqref{i1}-\eqref{i3} now read as
\[
(\int \xi |\sin^3\frac{v}{2}\cos^{{2\lambda}-3}\frac{v}{2}|\, dx)\, (\int \xi |\sin^{{2\lambda}-1}\frac{v}{2}\cos\frac{v}{2}|\, dx)
\leq(\int \xi \sin^2\frac{v}{2}\cos^{{2\lambda}-2}\frac{v}{2}\, dx)\, (\int \xi \sin^{2\lambda}\frac{v}{2}\, dx)
\]
because of H\"older inequalities
\[
\int \xi |\sin^3\frac{v}{2}\cos^{{2\lambda}-3}\frac{v}{2}|\, dx
\leq(\int \xi \sin^2\frac{v}{2}\cos^{{2\lambda}-2}\frac{v}{2}\, dx)^\frac{{2\lambda}-3}{{2\lambda}-2}\, (\int \xi \sin^{2\lambda}\frac{v}{2}\, dx)^{\frac{1}{{2\lambda}-2}}
\]
and
\[
(\int \xi |\sin^{{2\lambda}-1}\frac{v}{2}\cos\frac{v}{2}|\, dx)
\leq(\int \xi \sin^2\frac{v}{2}\cos^{{2\lambda}-2}\frac{v}{2}\, dx)^\frac{1}{{2\lambda}-2}\, (\int \xi \sin^{2\lambda}\frac{v}{2}\, dx)^{\frac{{2\lambda}-3}{{2\lambda}-2}}.
\]
We leave details to the reader since the argument is as same as the one in Section 2.
Now we know that
 \be\label{Energy5}\|Q(T)\|_{L^{\infty}},\ \  \|Q_x(T)\|_{L^{\infty}}\leq C_3(E_0,J_0(\T)),\ee
 for some positive constant $C_3(E_0,J_0(\T))$ and for any $T\in[0,\T]$.

Similarly, we can show that \eqref{PI2} and \eqref{QI2} also hold for any solutions of \eqref{S-1}-\eqref{S-2}, which tells us that $\|P\|_{L^2}$ and $\|Q_x\|_{L^2}$ are bounded when $T\in[0,\T]$. In fact, these bounds can be proved using  Lemma \ref{le_5.2} and the $L^\infty$ bounds of $P$ and $Q_x$

Plugging (\ref{Energy4}) and (\ref{Energy5}) into the last equation in (\ref{S-1}), we get
\[|\xi_T|\leq C_5(E_0,J_0(\T))\]
 for some positive constant $C_5(E_0,J_0(\T))$. Using the initial condition $\xi(0,Y)=1$, we have
\be\label{Bound10} \frac{1}{D_1}\leq \xi (T)\leq D_1,\ee  for some positive constant $D_1=D_1(E_0,J_0(\T))$ depending on $E_0$ and $J_0(\T)$.

By the secondly equation in (\ref{S-1}),
we have that
\[\frac{d}{dT}\|v(T)\|_{L^{\infty}}\leq \|u(T)\|_{L^{\infty}}^{{(\lambda-1)}}+2 \|u(T)\|_{L^{\infty}}^{\lambda+1}+\|P+Q_x\|_{L^{\infty}}\leq D_2\]
for some constant $D_2=D_2(E_0,J_0(\T))$. Hence,
\be\label{Bound1}  \|v(T)\|_{L^{\infty}} \leq e^{D_2 T}.\ee

\paragraph{\bf 3.}
Now we estimate $\|u(T)\|_{W^{1,{2\lambda}}}$. By the first equation of \eqref{S-1}, we have
\be\label{1}\begin{split} \frac {d}{dT}\| u(T)\|_{L^{2 \lambda }}^{2 \lambda }&=2 \lambda \int u^{2\lambda-1} u_TdY\\
  &\leq 2 \lambda  \|u(T)\|_{L^{\infty}}^{2\lambda-1}(\|P_x\|_{L^1}+ \| Q\|_{L^1}),
  \end{split}\ee
and
\be\label{2}\begin{split}\frac {d}{dT}\| (u(T))_Y\|_{L^{2 \lambda }}^{2 \lambda }& =2 \lambda \int u_Y^{2\lambda-1} (u_T)_YdY\\
& \leq  \frac{\lambda}{2^{2{(\lambda-1)}}} \|\xi\|_{L^{\infty}}^{2\lambda-1}(\|(P_x)_Y\|_{L^1}+ \|Q_Y\|_{L^1}).\\
\end{split}\ee
Now introduce a technical lemma very similar to Lemma \ref{le_4.2}, where the proof can be found in the Appendix B.
\begin{lemma}\label{le_5.2} Let $f\in L^q$ with $q\geq 1$. If $\|\xi \sin^2\frac {v(T)}{2}\cos^{2{(\lambda-1)}}\frac {v(T)}{2} \|_{L^1}\leq E_0$, then
\be\label{Singular2} \|\int_{-\infty}^{+\infty}  e^{-|\int_{Y}^{Y^{'}}\cos^{2\lambda}\frac{v(s)}{2} \cdot\xi(s)ds|} f(Y^{'})d Y^{'}\|_{L^q}\leq  \|h*f\|_{L^q}\leq\|h\|_{L^1} \ \|f\|_{L^q},\ee where $h(z)=\min\{1, e^{\frac{\alpha}{D_1}(  \frac{18  \ 4^{(\lambda-1)}}{( 2-\sqrt{3} )^{{(\lambda-1)}}}D_1E_0-|z|)}\}$ and $\|h\|_{L^1}=\frac{36  \ 4^{(\lambda-1)}}{( 2-\sqrt{3} )^{{(\lambda-1)}}}D_1E_0+\frac{2D_1}{\alpha}$.
\end{lemma}
Using (\ref{Singular2}), we know that
\[\begin{split}\|P_x\|_{L^1}&\leq \frac12 \|h\|_{L^1}\|(\lambda - \frac 12)\xi u^{{(\lambda-1)}} \sin^2\frac {v(T)}{2}\cos^{2{(\lambda-1)}}\frac {v(T)}{2}+\xi u^{\lambda+1}\cos^{2 \lambda } \frac  {v(T)}{2} \|_{L^1}\\
&\leq \frac12   \|h\|_{L^1} (\|u(T)\|_{L^{\infty}}^{{(\lambda-1)}}\|\xi  u^2(T)\cos^{2 \lambda } \frac {v(T)}{2}\|_{L^1}\\
&\qquad +\frac{2\lambda-1}{4}\|u(T)\|_{L^{\infty}}^{{(\lambda-1)}} \|\xi \sin^2\frac {v(T)}{2}\cos^{2{(\lambda-1)}}\frac {v(T)}{2} \|_{L^1})\\
&\leq\frac{2 \lambda+1}{4} E_0^{\frac{\lambda+1}{2}} \|h\|_{L^1},
\end{split}\]

\[\begin{split}\|(P_x)_Y\|_{L^1}&\leq\frac12 \|(\lambda - \frac 12)\xi u^{{(\lambda-1)}} \sin^2\frac {v(T)}{2}\cos^{2{(\lambda-1)}}\frac {v(T)}{2}+\xi u^{\lambda+1}\cos^{2 \lambda } \frac  {v(T)}{2} \|_{L^1}\\
&\qquad + \frac12 \|\xi \cos^{2 \lambda } \frac {v(T)}{2} P \|_{L^1}\\
&\leq \frac12  (1+\|\xi\|_{L^{\infty}} \|h\|_{L^1})\\
&\qquad \cdot\|(\lambda - \frac 12)\xi u^{{(\lambda-1)}} \sin^2\frac {v(T)}{2}\cos^{2{(\lambda-1)}}\frac {v(T)}{2}+\xi u^{\lambda+1}\cos^{2 \lambda } \frac  {v(T)}{2} \|_{L^1}\\
&\leq  \frac{2 \lambda+1 }{4} (1+D_1 \|h\|_{L^1})E_0^{\frac{\lambda+1}{2}} ,
\end{split}\]

\[\| Q\|_{L^1}\leq\frac{{(\lambda-1)}}{4}   \|u(T)\|_{L^{\infty}}^{ \lambda-2}\|e^{-|\int_{Y}^{Y^{'}} \cos^{2 \lambda }\frac v2 \cdot \xi ds|} \xi(T)\|_{L^{1}}\leq\frac{{(\lambda-1)}}{4} E_0^{\frac{\lambda-2}{2}}\|h\|_{L^1},
\]

\[\| Q_Y\|_{L^1}\leq \|\xi \ \cos^{2 \lambda }\frac v2 Q_x\|_{L^1}
\leq \frac{{(\lambda-1)} D_1}{4} E_0^{\frac{\lambda-2}{2}}\|h\|_{L^1}.
\]
Using above four estimates and (\ref{1})-(\ref{2}), we know that
\be\label{3}  \frac {d}{dT}\| u(T)\|_{L^{2 \lambda }}^{2 \lambda } \leq
2 \lambda E_0^{(\lambda - \frac 12)}\|h\|_{L^1}(\frac{2 \lambda+1 }{2}E_0^{\frac{ \lambda+1}{2}}+\frac{{(\lambda-1)} }{4} E_0^{\frac{\lambda-2}{2}}),\ee
\be\label{4}\begin{array}{lll} \frac {d}{dT}\| (u(T))_Y\|_{L^{2 \lambda }}^{2 \lambda } \leq
\frac{ \lambda D_1^{2\lambda-1}}{2^{2{(\lambda-1)}}}[\frac{2 \lambda+1 }{2}(1+D_1\|h\|_{L^1})E_0^{\frac{\lambda+1}{2}}+ \frac{{(\lambda-1)} D_1}{4} E_0^{\frac{\lambda-2}{2}}\|h\|_{L^1}].
  \end{array}\ee
 Then,  from (\ref{Singular2}) in Lemma 5.2, we can deduce that $\| u(T)\|_{L^{2 \lambda }}^{2 \lambda }$ and $\| (u(T))_Y\|_{L^{2 \lambda }}^{2 \lambda }$ are bounded on any bounded interval of time $T$.    That is, there exists a constant $D_3$ (depends on ${\lambda}$, $\|h\|_{L^1}$, $E_0$ and $D_1$) such that
\be\label{Bound3}  \|u(T)\|_{W^{1,{2\lambda}}}  \leq D_3.\ee

\paragraph{\bf 4.}
Finally, we estimate $\|v(T)\|_{L^{2}}$.
\[\begin{split}&\frac{d}{dT}\|v(T)\|_{L^{2}}^2\\
=&2\int_{\mathbb{R}} (-u^{{(\lambda-1)}}\sin^2\frac v2+2u^{\lambda+1}\cos^2\frac v2-2\cos^2\frac v2(P  +Q_x) v dY\\
\leq &\frac 12  \|u(T)\|_{L^{\infty}}^{{(\lambda-1)}} \|v(T)\|_{L^{\infty}}^{{(\lambda-1)}} \|v(T)\|_{L^{2}}^2+4 \|u(T)\|_{L^{\infty}}
\| u(T)\|_{L^{2 \lambda }}^{\lambda} \|v(T)\|_{L^{2}} +4 \|P+Q_x\|_{L^2}\|v(T)\|_{L^{2}}.
\end{split}\]
As we know that $\|u(T)\|_{L^{\infty}}$,  $ \|v(T)\|_{L^{\infty}}$,  $\| u(T)\|_{L^{2 \lambda }}$, $\|P\|_{L^2}$ and $\|Q_x\|_{L^2}$ are bounded when $T\in[0,\T]$. Then the bound of $ \|v(T)\|_{L^{2}}^2$ can be obtained by Gronwall's inequality.

\smallskip
\paragraph{\bf 5.}
Now we have proved that there exists some positive constant $C(E_0,J_0(\T))$ such as \eqref{B} holds. Then using a standard fixed point argument as in \cite{BC2}, one can uniquely extend the local solution to $[0,\T]$ for any $\T$, which completes the proof of this proposition. We refer the reader to \cite{BC2} for more details.
\end{proof}

\renewcommand{\theequation}{\thesection.\arabic{equation}}
\setcounter{equation}{0}
\section{Global existence of the original  equation}

We now show that the global solution of the semi-linear system (\ref{S-1}) found in the previous section yields a global conservative solution to equation (\ref{E}), in the original $(t,x)$ coordinates. We only have to prove this result when $t=T\in[0,\T]$ for any $\T>0$. We will finally finish the proof of Theorem \ref{main}.

  \begin{proof}
We start from a global solution $(u,v,\xi)$ to (\ref{S-1}) obtained in Proposition \ref{pro_5.1}. We define $t$ and $x$ as functions of $T$ and $Y$, where $t=T$ and
\be\label{5.1}x(T,Y):=x_0(Y)+\int_0^Tu^{\lambda}(\tau,Y)d\tau.\ee
For each fixed $Y$, the function $T\mapsto x(T,Y)$ thus provides a solution to the Cauchy problem
\[\frac{d}{dT} x(T,Y)=u^{\lambda}(T,x(T,Y)), \ \ \ x(0,Y)=x_0(Y).\]

In order to define $u$ as a function of the original variables $t$, $x$, we should formally invert the map $(T,Y)\mapsto  (t,x)$ and write
\be\label{5.2}u(t,x):=u(T,Y(t,x)),\quad \hbox{with}\ t=T.\ee We will prove that $u(t,x)$ is a weak conservative solution of (\ref{E0})(\ref{Ei}) under Definition \ref{def1}.

Since the map $(T,Y)\mapsto  (t,x)$  is not one to one,
 we first need to  prove that the function $u=u(t,x)$ is well-defined, i.e. if $x(T^*,Y_1)=x(T^*,Y_2)$, then $u(T^*,Y_1)=u(T^*,Y_2)$. By (\ref{Energy3}), we know
$|u(t,Y)|^{\lambda}\leq E_0^{\frac{\lambda}{2}}$. And by (\ref{5.1}), we get
\[x_0(Y)-E_0^{\frac{\lambda}{2}}t\leq x(t,Y)\leq x_0(Y)+E_0^{\frac{\lambda}{2}}t.\]
By  $Y=\int_0^{x_0(Y)} (1+(u_0)_x^2)^{{\lambda}}dx$, we have $\lim\limits_{Y\rightarrow \pm\infty} x_0(Y) =\pm\infty$, which yields that the image of the continuous map $(t,Y)\mapsto(t,x(t,Y))$
is the entire half plane $(t,x)\in\mathbb{R}^+\times \mathbb{R}$. On the other hand, we claim that
\be\label{5.4} x_Y=\cos^{{2\lambda}} \frac{v}{2} \cdot \xi \ee for any $t\geq 0$ and a.e. $Y\in \mathbb{R}$.
Indeed, from the last two equations in (\ref{S-1}), we know that
\[\begin{split}\frac{d}{dT} (\cos^{{2\lambda}} \frac{v}{2} \cdot \xi)
=&\frac{\lambda}{2}u^{{(\lambda-1)}} \xi \sin v \cos^{{2\lambda}-2} \frac v2
= \lambda u^{{(\lambda-1)}} u_Y.
\end{split}\]
Differentiating (\ref{5.1}) with respect to $T$ and $Y$, we know that
\[ \frac{d}{dT} x_Y=\frac{d}{dY} x_T= \lambda u^{{(\lambda-1)}} u_Y=\frac{d}{dT} (\cos^{{2\lambda}} \frac{v}{2} \cdot \xi).\]
Moreover, from the fact that $x\mapsto 2\arctan (u_0)_x(x)$ is measurable, we know that the claim (\ref{5.4})  is true for almost every $Y\in \mathbb{R}$ at $T=0$. Then, (\ref{5.4}) remains true for all times $T\in \mathbb{R}^+$ and a.e. $Y\in \mathbb{R}$.  Thus, for any $Y_1\neq Y_2$ (without loss of generality, we assume $Y_1<Y_2$), if $x(T^*, Y_1)=x(T^*, Y_2)$, then, by the monotonicity of $x(T,Y)$ on $Y$,
for any $Y\in [Y_1,Y_2]$, we have $x(T^*, Y )=x(T^*, Y_1)$. And,
by (\ref{5.4}) we get
\[0=x(T^*, Y_1)-x(T^*, Y_2)=\int_{Y_1}^{Y_2} x_Y(T^*,Y)dY=\int_{Y_1}^{Y_2} \cos^{{2\lambda}} \frac{v(T^*,Y)}{2} \cdot \xi(T^*,Y)dY.\]
Then, $\cos \frac{v(T^*,Y)}{2} = 0$ for almost any $Y\in [Y_1,Y_2]$, which tells that, also by  (\ref{uY}).
\[\begin{split}u(T^*, Y_1)-u(T^*, Y_2)&=\int_{Y_1}^{Y_2} u_Y(T^*,Y)dY\\
&=\int_{Y_1}^{Y_2}  \xi(T^*,Y) \sin \frac{v(T^*,Y)}{2} \cos^{2\lambda-1} \frac {v(T^*,Y)}{2}dY=0,\end{split}\]
where we use \eqref{uY}.
This proves that the map $(t,x)\mapsto u(t(T),x(T,Y))$ defined by (\ref{5.2}) is well defined for all $(t,x)\in \mathbb{R}^+\times\mathbb{R}$.

Secondly, we prove the regularity of $u(t,x)$ and energy equation.
By  (\ref{Energy2}), we know that, when $\lambda\geq 2$,
\be\label{5.7}\begin{split}
\mathcal{E}(0)&=E(0)=E(T)\\
=&\int_{\mathbb{R}}[u^2(T,Y)\cos^{2 \lambda } \frac {v(T,Y)}{2}+\sin^2\frac {v(T,Y)}{2}\cos^{2{(\lambda-1)}}\frac {v(T,Y)}{2}]\ \xi (T,Y)\,dY\\
=&\int_{\{\mathbb{R}\cap \cos\frac{v}{2}\neq 0\}}[u^2(T,Y)\cos^{2 \lambda } \frac {v(T,Y)}{2}+\sin^2\frac {v(T,Y)}{2}\cos^{2{(\lambda-1)}}\frac {v(T,Y)}{2}]\ \xi (T,Y)\,dY\\
=& \int_{\mathbb{R}} (u^2(t,x)+u_x^2(t,x))dx= \mathcal{E}(T).\end{split}\ee

When $\lambda=1$ (Camassa-Holm), we only can show $\mathcal{E}(T)\leq  \mathcal{E}(0)$, since $\int_{\mathbb{R}}\sin^2\frac {v(T,Y)}{2} \xi (T,Y)dY$ might be larger than 
 $\int_{\{\mathbb{R}\cap \cos\frac{v}{2}\neq 0\}}\sin^2\frac {v(T,Y)}{2} \xi (T,Y)dY$. Furthermore,  $\mathcal{E}(T)=  \mathcal{E}(0)$ for almost any time $t\geq 0$,  where the proof can be found in \cite{BC2}, or directly from the fact that the measure $\mu_{(t)}$ is absolutely continuous with respect to Lebesgue measure for almost any time when $\lambda=1$, which will be proved later.

Now, applying the Sobolev inequality and also using (\ref{j0t}), we get $$\|u\|_{C^{0,1-\frac{1}{{2\lambda}}}}\leq C \|u\|_{W^{1,{2\lambda}}}\leq J_0(\T),$$ which implies that for any $t$, $u(t,\cdot)$ is   uniformly H\"{o}lder continuous  on $x$ with exponent $1-\frac{1}{{2\lambda}}$. On the other hand, it follows from the first equation in (\ref{S-1}), (\ref{Energy4}) and (\ref{Energy5})  that $$\|u_T\|_{L^{\infty}}\leq C(\|P_x\|_{L^{\infty}}+\|Q\|_{L^{\infty}})\leq C(E_0).$$ Then, the map $t\mapsto  u(t,x(t))$ is uniformly Lipschitz continuous along every characteristic curve $t\mapsto x(t)$. Therefore, for any $(t,x)\in {\mathbb R}^+\times\mathbb R$, $u=u(t,x)$ is  H\"{o}lder continuous with respect to $t$ and $x$  with exponent $1-\frac{1}{{2\lambda}}$.

Thirdly, we prove that the $L^{2\lambda}(\mathbb{R})$-norm of $u(t)$  is Lipschitz continuous with respect to $t$.
Denote $[\tau, \tau+h]$ be any small interval.  For a given point $(\tau,\bar x)$, we choose the
characteristic $t\mapsto x(t,Y)$ $:\{T\rightarrow x(T,Y)\}$ passes through the point $(\tau, \bar x)$, i.e. $x(\tau) =\bar x$. Since
the characteristic speed $u^{\lambda}$ satisfies $\|u^{\lambda}\|_{L^{\infty}}\leq CE_0^{\frac{\lambda}{2}}$, we have the following estimate.
\[\begin{split}&|u(\tau+h,\bar x)-u(\tau,\bar x)|\\
\leq& |u(\tau+h,\bar x)-u(\tau+h,x(\tau+h,Y))|+|u(\tau+h,x(\tau+h,Y))-u(\tau,\bar x)|\\
\leq & \sup\limits_{|y-x|\leq E_0^{\frac{\lambda}{2}}h} |u(\tau+h,y)-u(\tau+h,\bar x)| +\int_{\tau}^{\tau+h} |P_x(t,Y)+Q(t,Y)|dt
\end{split}\]
Then,   we use the following inequality, $\forall \ {\lambda}\geq 1 $, \[ (\int_a^b f(x)dx )^{2\lambda}=\int_a^b f(x)dx\cdot \int_a^b f(x)dx \cdots \int_a^b f(x)dx   \leq (b-a)^{{2\lambda}-1}\|f\|_{L^{2\lambda}}^{2\lambda},\]
and deduce that \be\label{L2}\begin{split}
&\int_{\mathbb{R}}|u(\tau+h,\bar x)-u(\tau,\bar x)|^{2\lambda}dx\\
\leq & C2^{{2\lambda}-1}  \int_{\mathbb{R}}(\int_{\bar x-E_0^{\frac{\lambda}{2}}h}^{\bar x+E_0^{\frac{\lambda}{2}}h} |u_x(\tau+h,y)|dy)^{2\lambda}dx\\
&\qquad \quad +2^{{2\lambda}-1}\int_{\mathbb{R}}(\int_{\tau}^{\tau+h} |P_x(t,Y)+Q(t,Y)|dt)^{2\lambda}\cos^{{2\lambda}} \frac{v(t,Y)}{2} \cdot \xi(t,Y)dY\\
\leq &C2^{4\lambda-1}E_0^{\lambda^2}h^{{2\lambda}}  \|u_x(\tau+h,y)\|_{L^{2\lambda}}^{2\lambda}
+2^{{2\lambda}-1}h^{{2\lambda}-1}\|\xi(\tau)\|_{L^\infty} \int_{\tau}^{\tau+h}\|P_x +Q \|_{L^{2\lambda}}^{2\lambda}dt\\
\leq & C(E_0)h^{2\lambda}.
\end{split}\ee
Then,
this implies that the map $t\mapsto u(t)$ is Lipschitz continuous, in terms of the $x$-variable in $L^{2\lambda}(\mathbb{R})$.

Fourthly, we prove  that the function $u$ provides a weak solution of (\ref{E}). Now, denote
\[\Gamma:=[0,+\infty)\times \mathbb{R},\ \ \ \ \  \ \ \ \bar \Gamma :=\Gamma\cap\{(T,Y)|\cos \frac {v(T,Y)}{2}\neq 0\}.\]
We remark the fact that  by (\ref{uY}), we get $u_Y=0$ when $\cos \frac {v(T,Y)}{2}= 0$, which is given by (\ref{5.7}). Thus, for any test function $\phi(t,x)\in C_c^1(\Gamma)$, the first equation of (\ref{S-1}) has the following weak form.
\[\begin{split}0&=\int\int_{\Gamma}u_{TY}\phi  +[-(\lambda - \frac 12)u^{{(\lambda-1)}}\sin^2 \frac v2\cos^{{2\lambda}-2}\frac v2-u^{\lambda+1}\cos^{2\lambda} \frac v2+(P+Q_x)\cos^{2\lambda}\frac v2]\xi \phi\, dYdT\\
&=\int\int_{\Gamma} -u_{Y}\phi_T +[-(\lambda - \frac 12)u^{{(\lambda-1)}}\sin^2 \frac v2\cos^{{2\lambda}-2}\frac v2-u^{\lambda+1}\cos^{2\lambda} \frac v2+(P+Q_x)\cos^{2\lambda}\frac v2]\xi \phi \, dYdT\\
&\qquad+\int_{\mathbb{R}}(u_0)_x \phi(0,x)\, dx\\
&=\int\int_{\bar\Gamma} -u_{Y}\phi_T +[-(\lambda - \frac 12)u^{{(\lambda-1)}}\sin^2 \frac v2\cos^{{2\lambda}-2}\frac v2-u^{\lambda+1}\cos^{2\lambda} \frac v2+(P+Q_x)\cos^{2\lambda}\frac v2]\xi \phi \, dYdT\\
&\qquad+\int_{\mathbb{R}}(u_0)_x \phi(0,x)\, dx\\
&=\int\int_{ \Gamma} -u_{x}\phi_T +[-(\lambda - \frac 12)u^{{(\lambda-1)}}u_x^2-u^{\lambda+1} + P+Q_x]  \phi \,dxdt+\int_{\mathbb{R}}(u_0)_x \phi(0,x)\, dx\\
&=\int\int_{ \Gamma} -u_{x}(\phi_t+u^{\lambda}\phi_x)+[-(\lambda - \frac 12)u^{{(\lambda-1)}}u_x^2-u^{\lambda+1} + P+Q_x]  \phi \,dxdt+\int_{\mathbb{R}}(u_0)_x \phi(0,x)\, dx\\
\end{split}\]
which proves (\ref{nv_weak}), i.e. $u$ is a weak solution of \eqref{E0}.

Now, we introduce the Radon measures $\{\mu_{(t)}, t\in \mathbb{R}^+\}$: for any Lebesgue measurable set $\{x\in \mathcal{A}\} $ in $\mathbb{R}$, supposing the corresponding preimage set of the transformation is $\{ Y\in \mathcal{G}(\mathcal{A})\}$, we have
\[\mu_{(t)} (\mathcal{A})=\int_{ \mathcal{G}(\mathcal{A})} \xi \sin^{2\lambda}\frac v2 (t,Y)dY.\]
Note for each $t$, $u(t,\cdot)\in W^{1,{2\lambda}}$, the set $\{x\in \mathbb{R}|\cos\frac v2(t,x)=0\}$ is a Lebesgue measure zero set on $x$. Hence, for every $t\in \mathbb{R}^+$, the absolutely continuous part of $\mu_{(t)}$ w.r.t. Lebesgue measure has density $u_x^{2\lambda}(t,\cdot)$ by (\ref{5.4}).

It follows from (\ref{S-1}) that for any test function $\phi(t,x)\in C_c^1(\Gamma)$,
\[\begin{split}
&-\int_{\mathbb{R}^+}\Big\{\int (\phi_t+u^{\lambda}\phi_x)d\mu_{(t)}\Big\}dt\\
=&-\int\int_{\Gamma}\phi_T\xi\sin^{2\lambda}\frac v2dYdT\\
=& \int\int_{\Gamma}\phi (\xi\sin^{2\lambda}\frac v2)_TdYdT-\int_{\mathbb{R}}(u_0)^{2\lambda}_x \phi(0,x)\, dx\\
=& \int\int_{\bar\Gamma}{2\lambda}\phi\xi (u^{\lambda+1}-P-Q_x)\cos \frac v2 \sin^{{2\lambda}-1}\frac v2dYdT-\int_{\mathbb{R}}(u_0)^{2\lambda}_x \phi(0,x)\, dx\\
=& \int\int_{\Gamma}[{2\lambda}u^{\lambda+1}u_x^{{2\lambda}-1}-{2\lambda}u_x^{{2\lambda}-1}(P+Q_x)] \phi\, dxdt-\int_{\mathbb{R}}(u_0)^{2\lambda}_x \phi(0,x)\, dx.\\
\end{split}\]
Then, (\ref{weak_en}) is true. Moreover, in terms of the proof in \cite{BC2}, we see that for almost $t\in \mathbb{R}^+$, the singular part of  $\mu_{(t)}$ concentrates  on $u=0$. Indeed, if the solution blows up, then $|u_x|\rightarrow +\infty$, which means $\cos \frac v2=0$. Then, from the second identity in (\ref{S-1}), we get $v_T=-u^{(\lambda-1)}$, which is nonzero (hence $v=\pi$ is transversal): when $\lambda\geq 2$ and $u$ is nonzero; or when $\lambda=1$. 

Hence, by a similar proof as in \cite{BC2,CCL}, we can prove that for almost every $t\in \R^+$, the singular part of $\mu_{(t)}$: is concentrated on the set where $u=0$ when $\lambda\geq 2$; or has zero measure when $\lambda= 1$, i.e. $\mu_{(t)}$ is absolutely continuous with respect to the Lebesgue measure.

Finally, let $u_{0,n}$ be a sequence of initial data converging to $u_0$ in $H^1(\mathbb{R})\cap W^{1,{2\lambda}}(\mathbb{R})$. From (\ref{2.1}) and (\ref{S-2}), at time $t=0$ this implies
\[\sup_{Y\in \mathbb{R}}|x_n(0,Y)-x(0,Y)|\rightarrow 0,\ \   \ \  \sup_{Y\in \mathbb{R}}|u_n(0,Y)-u(0,Y)|\rightarrow 0, \]
Meanwhile, $\|v_n(0,\cdot)-v(0,\cdot)\|_{L^{2\lambda}}\rightarrow 0$. This implies that $u_n(T,Y)\rightarrow u(T,Y)$, uniformly for $T,Y$ in bounded sets. Returning to the original coordinates, this yields the convergence
\[x_n(T,Y)\rightarrow x(T,Y),\ \ \ \  u_n(t,x)\rightarrow u(t,x),\]
uniformly on bounded sets, because all functions $u$, $u_n$ are uniformly H\"{o}lder  continuous.

\end{proof}

\section*{Acknowledgments}
The first author is partially supported by NSF with grants DMS-1715012 and DMS-2008504.

\appendix
\section{ Proof of Lemma \ref{le_4.2}}
\begin{proof}
 Indeed, when $|v(Y)|\leq \frac{3\pi}{2}$,  we see that
\be\label{Singular-1}|\frac{v(Y)}{2}|\leq \frac{3\pi}{4},\ \  \sin^2\frac{v(Y)}{2}\leq (\frac{v(Y)}{2})^2\leq \frac{9\pi^2}{8}\sin^2\frac{v(Y)}{2}.\ee
Thus, when $|\frac{v(Y)}{2}|\geq \frac{\pi}{4}$, we have $\sin^2\frac{v(Y)}{2}\geq \frac{8}{9\pi^2} |\frac{v(Y)}{2}|^2\geq \frac{1}{18}$. And
for any $(u,v,\xi)\in K$, we deduce that
\be\label{Measure0}\begin{split}& {\rm meas}\{Y\in \mathbb{R}|\ |\frac{v(Y)}{2}|\geq \frac{\pi}{4}  \}\\
\leq & {\rm meas} \{Y\in \mathbb{R}| \sin^2\frac{v(Y)}{2}\geq \frac{1}{18}\}\\
\leq& 18 \int_{\{Y\in \mathbb{R}| \sin^2\frac{v(Y)}{2}\geq \frac{1}{18}\}} \sin^2\frac{v(Y)}{2}dY\\
\leq&   \frac{9  }{2}\|v\|_{L^2}^2\leq  \frac{9  }{2}B^2.
\end{split}\ee
For any $z_1<z_2$, we can show that,
\[\begin{split}\int_{z_1}^{z_2}\cos^{2\lambda}\frac{v(z)}{2} \cdot\xi(z)dz
\geq \frac{ C^-}{2^{\lambda}}\big(|z_2-z_1|- \int_{\{ z\in[z_1,z_2]|\  |\frac{v(Y)}{2}|\geq \frac{\pi}{4} \}}dz\big)\geq  \frac{ C^-}{2^{\lambda}}\big(|z_2-z_1|-  \frac{9 }{2}B^2 \big).\end{split} \]
Therefore, for every $q\geq 1$, we see that
\[ |\int_{-\infty}^{+\infty}  e^{-|\int_{Y}^{Y^{'}}\cos^{2\lambda} \frac{v(s)}{2} \cdot\xi(s)ds|} f(Y^{'})d Y^{'}|\leq |g*f(Y)|.\]
Then (\ref{Singular}) follows from the Young inequality. Moreover, we get that
\[\|g\|_{L^1}=\int_{-  \frac{9   }{2 }B^2 }^{ \frac{9  }{2 }B^2 } 1 dz+\int_{ \frac{9  }{2}B^2}^{+\infty} e^{  \frac{C^-}{2^{\lambda}}(\frac{9   }{2 }B^2- z) }dz  +\int^{- \frac{9 }{2 }B^2}_{-\infty} e^{  \frac{C^-}{2^{\lambda}}( \frac{9   }{2 }B^2+ z) }dz= 9B^2+\frac{2^{\lambda+1}}{ C^-}.\]
\end{proof}

\section{ Proof of Lemma \ref{le_5.2}}
\begin{proof}

As in Lemma 3.1,when $\frac{\pi}{4}\leq |\frac{v(Y)}{2}|\leq  \frac{5\pi}{12}  $, we deduce that $\cos \frac{5\pi}{12}=\frac{\sqrt{6}-\sqrt{2}}{4}$ and for ${(\lambda-1)}\in \mathbb{Z}^{+}$
\[\cos^{2{(\lambda-1)}} \frac{v(Y)}{2}\sin^2\frac{v(Y)}{2}\geq \frac{8}{9\pi^2} (\frac{v(Y)}{2})^2(\frac{2-\sqrt{3}}{4})^{{(\lambda-1)}}\geq\frac{( 2-\sqrt{3} )^{{(\lambda-1)}}}{18\ 4^{(\lambda-1)}}.\]
Then we obtain the following measure estimation.
\[\begin{split}& {\rm meas}\{Y\in \mathbb{R}|\frac{\pi}{4}\leq |\frac{v(Y)}{2}|\leq \frac{5\pi}{12}\}\\
\leq & {\rm meas} \{Y\in \mathbb{R}|\cos^{2{(\lambda-1)}} \frac{v(Y)}{2}\sin^2\frac{v(Y)}{2}\geq \frac{( 2-\sqrt{3} )^{{(\lambda-1)}}}{18 \ 4^{(\lambda-1)}}\}\\
\leq& \frac{18\ 4^{(\lambda-1)}}{( 2-\sqrt{3} )^{{(\lambda-1)}}}\int_{\{Y\in \mathbb{R}|\cos^{2{(\lambda-1)}} \frac{v(Y)}{2}\sin^2\frac{v(Y)}{2}\geq \frac{( 2-\sqrt{3} )^{{(\lambda-1)}}}{18\ 4^{(\lambda-1)}}\}}D_1\ \xi  \cos^{2{(\lambda-1)}} \frac{v(Y)}{2}\sin^2\frac{v(Y)}{2}dY\\
\leq&   \frac{18  \ 4^{(\lambda-1)}}{( 2-\sqrt{3} )^{{(\lambda-1)}}}D_1E_0.
\end{split}\]
 For any $z_1<z_2$, we deduce that  there exists $\alpha>0$ such that
\[\begin{split}\int_{z_1}^{z_2}\cos^{2 \lambda }\frac{v(z)}{2} \cdot\xi(z)dz&\geq \frac{\alpha}{D_1} (
\int_{z_1}^{z_2}1 dz-\int_{\{ z\in [z_1,z_2]\ |\ \  \frac {\pi}{4}\leq |\frac{v(Y)}{2}| \leq \frac{5\pi}{12}\}} 1 dz)\\
&\geq \frac{\alpha}{D_1}(|z_2-z_1|-  \frac{18  \ 4^{(\lambda-1)}}{( 2-\sqrt{3} )^{{(\lambda-1)}}}  D_1E_0 ). \end{split}\]
The above estimate guarantees proper control on the singular integrals in $P$, $P_x$, $Q$ and $Q_x$, which decreases quickly as $|Y-Y^{'}|\rightarrow +\infty$. Therefore,
\[|\int_{-\infty}^{+\infty}  e^{-|\int_{Y}^{Y^{'}}\cos^{2\lambda}\frac{v(s)}{2} \cdot\xi(s)ds|} f(Y^{'})d Y^{'}|\leq |h*f(Y)|,\]
where $h(z)=\min\{1, e^{\frac{\alpha}{D_1}(  \frac{18  \ 4^{(\lambda-1)}}{( 2-\sqrt{3} )^{{(\lambda-1)}}}D_1E_0-|z|)}\}$. Then,  for every $q\geq 1$, (\ref{Singular2}) follows from the  Young inequality.
And $\|h\|_{L^1}$ has the following estimate.
\[\begin{split}\|h\|_{L^1}&=\int_{-  \frac{18  \ 4^{(\lambda-1)}}{( 2-\sqrt{3} )^{{(\lambda-1)}}}D_1E_0 }^{  \frac{18  \ 4^{(\lambda-1)}}{( 2-\sqrt{3} )^{{(\lambda-1)}}}D_1E_0 } 1 dz+\int_{ \frac{36  \ 4^{(\lambda-1)}}{( 2-\sqrt{3} )^{{(\lambda-1)}}}D_1E_0}^{+\infty} e^{\frac{  \alpha[\frac{18  \ 4^{(\lambda-1)}}{( 2-\sqrt{3} )^{{(\lambda-1)}}}D_1E_0- z]}{D_1}}dz\\
&\qquad+\int^{- \frac{18  \ 4^{(\lambda-1)}}{( 2-\sqrt{3} )^{{(\lambda-1)}}}D_1E_0}_{-\infty} e^{\frac{ \alpha[\frac{18  \ 4^{(\lambda-1)}}{( 2-\sqrt{3} )^{{(\lambda-1)}}}D_1E_0+ z]}{D_1}}dz\\
&= \frac{36  \ 4^{(\lambda-1)}}{( 2-\sqrt{3} )^{{(\lambda-1)}}}D_1E_0+\frac{2D_1}{\alpha}.\end{split}\]

\end{proof}
%


\begin{thebibliography}{40}

\bibitem{Anco}S. C. Anco, P. da Silva and I. Freire,
A family of wave-breaking equations generalizing the Camassa-Holm and Novikov equations,  J. Math. Phys., {\bf 56} (2015) 091506 (21 pages).



\bibitem{BC} A. Bressan and G. Chen,  
Generic regularity of conservative solutions to a nonlinear wave equation.  
{\em Ann. I. H. Poincar\'{e}--AN} {\bf  34} (2017),  no. 2, 335-354.

\bibitem{BCZ2015} A. Bressan, G. Chen and Q. Zhang, Uniqueness of conservative solutions to the Camassa-Holm equation via characteristics, Discrete Contin. Dyn. Syst., {\bf 35} (2015), 25--42.

\bibitem{BC2}
A.~Bressan and A.~Constantin,
Global conservative solutions to the Camassa-Holm equation,
Arch. Ration. Mech. Anal., {\bf 183} (2007), 215--239.

\bibitem{BC3}
A.~Bressan and A.~Constantin,
Global dissipative solutions to the Camassa-Holm equation,
 Anal. Appl., {\bf 5} (2007), 1--27.





\bibitem{BF}
A.~Bressan and M.~Fonte,
An optimal transportation metric for solutions of the
Camassa-Holm equation,  Methods Appl. Anal.,
{\bf 12} (2005), 191--220.

\bibitem{BZZ} \newblock A. Bressan, P. Zhang, and Y. Zheng: \newblock{Asymptotic
variational wave equations},  \newblock\emph{Arch. Ration. Mech. Anal.}, \textbf{183} (2007),  163--185.









\bibitem{CCCS} H. Cai, G, Chen, R. M. Chen and Y. Shen, Lipschitz metric for the Novikov equation, {\em Arch. Ration. Mech. Anal.}, {\bf 229} (3) (2018), 1091--1137.


\bibitem{CCM} H.~Cai, G.~ Chen and H.~Mei,
Uniqueness of dissipative solution for Camassa-Holm equation with peakon antipeakon initial data, Appl. math. lett., {\bf  120} (2021), 107268.


\bibitem{CH1993} R. Camassa and D. D. Holm, An integrable shallow water equation with peaked solitons, Phys. Rev. Lett.,  {\bf 71} (1993), 1661--1664.

\bibitem{cht} C. Cao, D. D. Holm, and E. S. Titi, Traveling wave solutions for a class of one-dimensional nonlinear shallow water wave models, J. Dynam. Differential Equations, {\bf 16} (2004), no. 1, 167--178.

\bibitem{Panos} E. G. Charalampidis, R. Parker, P. G. Kevrekidis and S. Lafortune, The Stability of the b-family of Peakon Equations, preprint, available
at arXiv:2012.13019.






\bibitem{CCL} G. Chen, M. Chen and Y. Liu, Existence and uniqueness of the global conservative weak solutions for the integrable Novikov equation, Indiana Univ. Math. J., {\bf 67} (2018), 2393--2433.

\bibitem{CS2015} G. Chen and Y. Shen, Existence and regularity of solutions in nonlinear wave equations, Discrete Contin. Dyn. Syst., {\bf 35} (2015), 3327--3342.




\bibitem{CE1998} A. Constantin and J. Escher, Wave breaking for nonlinear nonlocal shallow water equations, Acta Math., {\bf 181} (1998), 229--243.








\bibitem{CR} G.  M.  Coclite and L.  di Ruvo,  Well-posedness results for the short pulse equation,  {\it Z. Angew. Math. Phys.}, {\bf 66} (2015),  1529-1557.

\bibitem{D} C. Dafermos, Hyperbolic conservation laws in continuum physics,  3rd Ed., {\it Springer-Verlag,  Berlin},  2010.

\bibitem{FF1981} B. Fuchssteiner and A.S. Fokas, Symplectic structures, their B\"{a}cklund transformations and hereditary symmetries, Physica D, {\bf 4} (1981/1982), 47--66.

\bibitem{Glimm} (MR0194770)
\newblock J. Glimm,
\newblock \emph{Solutions in the large for nonlinear hyperbolic systems of equations}, 
\newblock Comm. Pure Appl. Math., \textbf{18}(1965), 697–715.

\bibitem{GHR} K. Grunert, H. Holden and X. Raynaud, A continuous interpolation between conservative and dissipative solutions for the two-component Camassa-Holm system,  Forum Math. Sigma, {\bf 3} (2015) 73 pages.


\bibitem{GHR2011}
K.~Grunert, H.~Holden and X.~Raynaud,
Lipschitz metric for the periodic Camassa-Holm equation,
 J. Differential Equations, {\bf  250} (2011), 1460--1492.

\bibitem{GHR2013}
K.~Grunert, H.~Holden and X.~Raynaud, Lipschitz metric for the Camassa-Holm equation on the line, Discrete Contin. Dyn. Syst., {\bf  33} (2013), 2809--2827.















\bibitem{HS}
\newblock J.~K.~Hunter and R.~H.~Saxton, \newblock{Dynamics of director fields}, 
\newblock\emph{SIAM J. Appl. Math.}, {\bf 51} (1991), 1498-1521.

\bibitem{HZ95a} \newblock J.~K.~Hunter and Yuxi Zheng,
\newblock {On a nonlinear hyperbolic variational equation: I. global existence
of weak solutions},  
\newblock \emph{Arch. Rat. Mech. Anal.}, {\bf 129} (1995),
305--353.

\bibitem{well} Y. A. Li and P. J. Olver, Well-posedness and blow-up solutions for an integrable nonlinearly dispersive model wave equation, J. Differential Equations, {\bf 162} (2000), no. 1, 27--63.

  \bibitem{LPS} 
Y. Liu,  D. Pelinovsky and A. Sakovich,  Wave breaking in the short-pulse equation,
{\em Dynamics of PDE,} {\bf 6} (2009),  no.4, 291-310.



\bibitem{M1998}H. P. McKean, Breakdown of a shallow water equation.
Mikio Sato: a great Japanese mathematician of the twentieth century,
Asian J. Math. {\bf 2} (1998), no. 4, 867--874.

\bibitem{N2009} V. Novikov, Generalizations of the Camassa-Holm type equation, J. Phys. A,  {\bf 42} (2009), 342002, 14 pp.


\bibitem{XZ2000} Z. Xin and P. Zhang, On the weak solutions to a shallow water equation, Comm. Pure Appl. Math., {\bf 53} (2000), 1411--1433.





















\end{thebibliography}
\end{document}